\documentclass[11pt,a4paper]{article}
\usepackage{amsfonts}
\usepackage{amsthm}
\usepackage{amsmath}
\usepackage{amssymb}
\usepackage{mathrsfs}
\usepackage{amscd}
\usepackage[utf8]{inputenc}
\usepackage{t1enc}
\usepackage{lmodern}
\usepackage{dsfont}
\usepackage{hyperref}
\usepackage[dvipsnames]{xcolor}
\hypersetup{colorlinks=true, urlcolor= NavyBlue, linkcolor=NavyBlue, citecolor=NavyBlue}
\usepackage[mathscr]{eucal}
\usepackage{indentfirst}
\usepackage{graphicx}
\usepackage{graphics}
\usepackage{pict2e}
\usepackage{epic}
\usepackage[margin=2cm]{geometry}
\usepackage{epstopdf} 
\usepackage[round]{natbib}

\RequirePackage{algorithm}
\RequirePackage{algorithmic}
\usepackage{caption}
\usepackage{indentfirst}
\usepackage{cleveref}

\usepackage{microtype}
\usepackage{graphicx}
\usepackage{subfigure}
\usepackage{booktabs} 
\usepackage{float}
\usepackage{xcolor}

\usepackage{array}

\newcommand{\esp}[1]{\mathbb{E}\left[#1\right]}
\newcommand{\espk}[2]{\mathbb{E}_{#1}\left[#2\right]}
\newcommand{\R}{\mathbb{R}}

\newcommand{\xit}{{\xi_t}}

\newcommand{\dom}{{\rm dom}\ }
\newcommand{\interior}{{ \rm int}\ }
\newcommand{\reals}{\R}

\usepackage{amsthm}

\makeatletter
\newtheorem*{rep@theorem}{\rep@title}
\newcommand{\newreptheorem}[2]{%
\newenvironment{rep#1}[1]{%
 \def\rep@title{#2 \ref{##1}}%
 \begin{rep@theorem}}%
 {\end{rep@theorem}}}
\makeatother

\newtheorem{definition}{Definition}[]
\newtheorem{assumption}{Assumption}[]
\newtheorem{theorem}{Theorem}[]
\newtheorem{proposition}{Proposition}[]
\newtheorem{lemma}{Lemma}[]
\newtheorem{remark}{Remark}[]
\newtheorem{corollary}{Corollary}[]

\newreptheorem{theorem}{Theorem}
\newreptheorem{lemma}{Lemma}
\newreptheorem{proposition}{Proposition}
\newreptheorem{corollary}{Corollary}

\usepackage{authblk}

\usepackage[english]{babel}

\author[*,1,3]{Radu-Alexandru Dragomir}
\author[*,2,3]{Mathieu Even}
\author[*,2,3]{Hadrien Hendrikx}

\affil[*]{Alphabetical order, equal contribution}
\affil[1]{Université Toulouse 1 Capitole}
\affil[2]{INRIA Paris}
\affil[3]{D.I. Ecole Normale Supérieure, CRNS, PSL University, Paris}
\date{}

\begin{document}
\title{\textbf{Fast Stochastic Bregman Gradient Methods:\\ Sharp Analysis and Variance Reduction}}

\maketitle
\begin{abstract}
 We study the problem of minimizing a relatively-smooth convex function using stochastic Bregman gradient methods. We first prove the convergence of Bregman Stochastic  Gradient Descent (BSGD) to a region that depends on the noise (magnitude of the gradients) at the optimum. In particular, BSGD with a constant step-size converges to the exact minimizer when this noise is zero (\emph{interpolation} setting, in which the data is fit perfectly). Otherwise, when the objective has a finite sum structure, we show that variance reduction can be used to counter the effect of noise. In particular, fast convergence to the exact minimizer can be obtained under additional regularity assumptions on the Bregman reference function. We illustrate the effectiveness of our approach on two key applications of relative smoothness: tomographic reconstruction with Poisson noise and statistical preconditioning for distributed optimization. 
\end{abstract}

\textbf{Keywords:}
  relative smoothness, Bregman gradient, mirror descent, stochastic methods, variance reduction, Poisson inverse problems, statistical preconditioning.%

\section{Introduction}
We are interested in solving the minimization problem
\begin{equation} \label{eq:main_intro}
    \min_{x \in C} f(x),\text{ where } f(x) = \espk{\xi}{f_\xi(x)},
\end{equation}
where $C$ is a closed convex subset of $\R^d$ and $f_\xi$ are differentiable convex functions. These problems typically arise in machine learning when performing (empirical) risk minimization, in which case $f_\xi$ is for instance a loss function for some random sample $\xi$. Problem~\eqref{eq:main_intro} is also encountered in signal processing applications such as image deblurring or tomographic reconstruction inverse problems, in which the goal is to recover an unknown signal from a large number of noisy observations. First-order methods are often very efficient for solving problems such as~\eqref{eq:main_intro}, but computing a gradient $\nabla f$ might be very expensive for large-scale problems (large number of components $f_\xi$), and even impossible in the case of true risk minimization (infinite number of $f_\xi$). In this case, stochastic gradient methods have proven to be particularly effective thanks to their low cost per iteration. The simplest one, Stochastic Gradient Descent (SGD), consists in updating $x_t$ as
\[ x_{t+1} = \arg \min_{x \in C}\left\{ \eta_t g_t^\top x +\frac{1}{2} \|x-x_t\|^2\right\} \]
where $g_t$ is a gradient estimate such that $\esp{g_t} = \nabla f(x_t)$. In our case, a natural choice would be $g_t = \nabla f_{\xi_t}(x_t)$ for some $\xi_t$. The choice of the step size $\eta_t$ is crucial for obtaining good performances and is typically related to the smoothness of $f$ with respect to the Euclidean norm.

Beyond simply adapting the step size, a powerful generalization of SGD consists in refining the geometry and performing instead Bregman gradient (a.k.a mirror) steps as
\begin{equation} \label{eq:iter_sgd}
    x_{t+1} = \arg \min_{x \in C}\left\{\eta_t \ g_t^\top x + D_h(x, x_t)\right\},
\end{equation}
where the Euclidean distance has been replaced by the Bregman divergence with respect to a reference function $h$, which writes:
\begin{equation} \label{eq:bregman_divergence}
        D_h(x, y) = h(x) - h(y) - \nabla h(y)^\top(x - y),
\end{equation}
for all $x \in \dom h, y \in \interior \dom h$. We make the following blanket assumptions on $h$ throughout the article, which guarantee well-posedness of the update~\eqref{eq:iter_sgd}.

\begin{assumption}\label{assumption:blanket}
The function $h:\R^n \rightarrow \R \cup \{\infty\}$ is twice continuously differentiable and strictly convex on $\interior C$. Moreover, for every $y \in \R^d$, the problem
\[ \min_{x \in C} h(x) - x^\top y\]
has a unique solution, which lies in $\interior C$.
\end{assumption}

The standard SGD algorithm corresponds to the case where $h = \frac{1}{2}\| \cdot \|^2$. However, a different choice of $h$ might better fit the geometry of the set $C$ and the curvature of the function, allowing the algorithm to take larger steps in directions where the objective gradient changes slowly. This choice is guided by the notion of relative smoothness and strong convexity, introduced in \citet{Bauschke2017,lu2018relatively}. Instead of the squared Euclidean norm for standard smoothness, \textit{relative} regularity is measured with respect to the reference function $h$.
\begin{definition}
    The function $f$ is said to be $L$-relatively smooth and $\mu$-relatively strongly convex with respect to $h$ if it is differentiable and for all $x,y \in \interior \dom h$, 
    \begin{equation} \label{eq:rel_smooth_sc}
        \mu D_h(x, y) \leq D_f(x,y) \leq L D_h(x, y).
    \end{equation}
\end{definition}
where $D_f$ is defined similarly to~\eqref{eq:bregman_divergence}.
Note that if $\mu = 0$, the left-hand side inequality reduces to assuming convexity of $f$. Similarly, if $h = \frac{1}{2}\| \cdot \|^2$, then $D_h(x, y) = \frac{1}{2}\| x - y\|^2$, and the usual notions of smoothness and strong convexity are recovered. If both functions are two times differentiable, Equation~\eqref{eq:rel_smooth_sc} can be turned into an equivalent condition on the Hessians: $\mu \nabla^2 h(x, y) \preceq \nabla^2 f(x) \preceq L \nabla^2 h(x)$. Throughout the article, we will generally write $\mu_{f/h}$ and $L_{f/h}$ to insist on the relative aspect.

Writing the optimality conditions for the minimization problem of Equation~\eqref{eq:iter_sgd}, we obtain the following equivalent iteration, which is in the alternative Mirror Descent form~\citep{nemirovsky1983problem}:
\begin{equation}
\label{eq:iter_smd_explicit}
    \nabla h(x_{t+1}) = \nabla h(x_t) - \eta_t g_t.
\end{equation} 
Although these updates have a closed-form solution for many choices of the reference function $h$, they may be harder to perform than standard gradient steps, since they require solving the subproblem defined in~\eqref{eq:iter_sgd}. Yet, this may be worth doing in some cases to reduce the overall iteration complexity, if the resulting majorization in \eqref{eq:rel_smooth_sc} is much tighter than with the Euclidean distance. Let us list some applications of relative regularity:\\

\textbf{Problems with unbounded curvature.} Some problems have singularities at some boundary points in $C$ where the Hessian grows arbitrarily large. In this situation, smoothness with respect to the Euclidean norm does not hold globally, and standard gradient methods become inefficient as they necessit excessively small step sizes or costly line search procedures. A typical example arises in inverse problems with Poisson noise, which are used in particular for image deblurring~\citep{Review2009} or tomographic reconstruction~\citep{kak2002principles}. In this case, the objective function involves the Kullback-Leibler divergence, which becomes singular as one of its arguments approaches 0. However, by choosing the reference function ${h(x) = -\sum_{i=1}^d \log(x^{(i)})}$, one can show that relative smoothness holds globally~\cite{Bauschke2017}. For more examples, see \citet{lu2018relatively,Bolte2018noncvx,Nesterov2019,mishchenko2019sinkhorn}.\\

\textbf{Distributed optimization.} When $h$ approximates $f$ in the sense of~\eqref{eq:rel_smooth_sc}, Bregman methods can be used to speed up convergence by performing non-uniform preconditioning~\citep{shamir2014communication, reddi2016aide, yuan2020convergence, hendrikx2020statistically}. Typically, $h$ is chosen as the objective function on a smaller portion of the dataset of size $n_{\rm prec}$ (\emph{e.g.}, the dataset of the server), which improves the conditioning by a factor of up to $n_{\rm prec}$ compared to Euclidean methods, while naturally taking advantage of an eventually small effective dimension of the dataset \citep{even2021concentration}. In this case, forming the gradient $g_t$ requires communication with the workers (where most of the data is held), and is thus expensive. Although the updates may not have a simple expression, the inner problem of Equation~\eqref{eq:iter_sgd} can be solved locally at the server without additional communications. Therefore, Bregman methods allow to drastically reduce the communication cost by reducing the overall iteration complexity.\\

Despite these applications, there are still many gaps in our understanding of convergence guarantees of Bregman gradient methods. In particular, most existing results focus on the deterministic case $g_t = \nabla f(x_t)$, or do not leverage the relative regularity assumptions. 

\subsection*{Contributions and outline}
In this work, we develop convergence theorems for Bregman SGD, for which the variance depends on the magnitude of the stochastic gradients at the optimum, and which can thus be much smaller than the one used in~\citet{Hanzely2018}, in particular for overparametrized models (which verify the interpolation condition that all stochastic gradients are equal to $0$ at the optimum). Our analysis relies on the Bregman generalization of a few technical lemmas such as the celebrated $\|a + b\|^2 \leq 2(\|a\|^2 + \|b\|^2)$ inequality (Lemma~\ref{lemma:bregman_young}) or the co-coercivity inequality (Lemma~\ref{lemma:bregman_cocoercivity_main}), which we believe to be of independent interest. 

Then, we show that variance-reduction techniques, which are widely used to accelerate traditional Euclidean stochastic methods when the objective has a finite-sum structure~\citep{SAG, johnson2013accelerating, Defazio2014, allen2017katyusha}, can be adapted to the Bregman setting. Although this generally requires stronger regularity assumptions (such as global smoothness of $h$ and Lipschitz continuity of $\nabla^2 h^*$), we show that the asymptotical rate of convergence solely depends on relative regularity constants. The same type of results (asymptotic speedup under additional smoothness assumptions) is observed when applying Nesterov-type acceleration to Bregman gradient methods ~\citep{Hanzely2018, dragomir2019optimal, hendrikx2020statistically}. We provide a summary of the rates proven in this paper in the appendix.

We start by discussing the related work in Section~\ref{sec:related_work}. Then, Section~\ref{sec:sgd} presents the results for stochastic gradient descent, along with the main technical lemmas. Section~\ref{sec:variance_reduction} develops a Bregman version of the standard SAGA algorithm~\citep{Defazio2014}. Finally, Section~\ref{sec:experiments} illustrates the efficiency of the proposed methods on several applications, including Poisson inverse problems, tomographic reconstruction and distributed optimization. 

\section{Related work}
\label{sec:related_work}
The Bregman gradient method was first introduced as the \textit{Mirror Descent} scheme\footnote{Note that \emph{Mirror Descent} and \emph{Bregman Gradient} refer to the same algorithm, but that \emph{Mirror Descent} is typically used when $f$ is non-smooth, or in the online optimization community, whereas \emph{Bregman Gradient} is generally preferred when using the relative smoothness assumption. Yet, both names are valid and there are exceptions, for instance~\citet{hanzely2018fastest} use the \emph{Mirror Descent} terminology although they assume relative smoothness.}~\citep{nemirovsky1983problem,beck2003mirror} for minimizing convex nonsmooth functions, and enjoyed notable success in online learning \cite{Bubeck2011}.
More recently, the introduction of relative smoothness~\citep{Bauschke2017, lu2018relatively, Bolte2018noncvx} has also brought interest in applying Bregman methods to differentiable objectives. This condition guides the choice of a well-suited reference function $h$ which can greatly improve efficiency over standard gradient descent. While the vanilla Bregman descent method yields the same convergence rate as the Euclidean counterpart, subsequent work has focused on obtaining better rates with acceleration schemes~\citep{Hanzely2018}. However, lower bounds show that the rates for relatively smooth optimization cannot be accelerated in general~\citep{dragomir2019optimal}, and that additional regularity assumptions are needed. Similar notions of relative regularity have also been investigated for non-differentiable functions, such as relative continuity~\citep{lu2019relative, antonakopoulos2019adaptive}. \citet{zhou2020regret} also study non-differentiable functions, but in the online setting and without relative continuity. 

Stochastic optimization methods, and in particular SGD, are very efficient when the number of samples is high~\citep{bottou2012stochastic} and are often referred to as ``the workhorse of machine learning''. The problem with SGD is that, in general, it only converges to a neighbourhood of the optimum unless a diminishing step-size is used. Variance reduction can be used to counter this problem, and many variance-reduced methods have been developed, such as SAG~\citep{SAG}, SDCA~\citep{shalev2013stochastic, shalev2016sdca}, SVRG~\citep{johnson2013accelerating} or SAGA~\citep{Defazio2014}. 

Surprisingly, stochastic Bregman gradients algorithms have received less attention. \citet{hanzely2018fastest, gao2020randomized,hendrikx2020dual} study Bregman coordinate descent methods, and \citet{zhang2018convergence} study the non-convex non-smooth setting. \citet{antonakopoulos2020online} study stochastic algorithms for online optimization, under Riemann-Lipschitz continuity. In contrast, our work focuses on Bregman SGD for relatively-smooth objectives. \citet{hanzely2018fastest} study the same setting and obtain comparable convergence rates, but with a much looser notion of variance, which we discuss more in details in the next section. This is problematic since their bound on the variance is thus proportional to the magnitude of the gradients along the trajectory, and may thus be very large when far from the optimum if $f$ is strongly convex. In contrast, our definition of variance leverages the stochastic gradients at the optimum, which allows us to obtain significant results without bounded gradients and in the interpolation regime (zero gradients at the optimum). In particular, our analysis can be seen as a Bregman generalization of the analysis from~\citet{gower2019sgd}. \citet{Davis2018StochasticMM} also analyze a similar setting, but again with more restrictive assumptions on the noise and boundedness of the gradients. Besides, to the best of our knowledge, variance reduction for Bregman stochastic methods was only studied in~\citet{shi2017bregman} in the context of stochastic saddle-point optimization, but without leveraging relative regularity assumptions like we do in this work.

\section{Bregman Stochastic Gradient Descent}
\label{sec:sgd}
\subsection{Preliminaries}
We start by introducing a few technical lemmas, which are Bregman analogs to well-known Euclidean results, and which are at the heart of our analysis. All missing proofs can be found in Appendix~\ref{app:sgd}.

Recall that the conjugate $h^*$ is defined for $y \in \R^d$ as $h^*(y) = \sup_{x \in \reals^d}  x^\top y - h(x)$. In particular, under Assumption \ref{assumption:blanket}, $h^*$ is convex and differentiable on $\reals^d$ \citep[Cor. 18.12]{Bauschke2011book}, and $\nabla h^*(\nabla h(y)) = y$ for $y \in \interior C$, which implies the following result:
\begin{lemma}[Duality] \label{lemma:duality}
    For $x,y \in \interior \dom h$, we have $D_h(x, y) = D_{h^*}(\nabla h(y), \nabla h(x))$.
\end{lemma}
See, \emph{e.g.},~\citet[Thm 3.7.]{Bauschke1997} for the proof. Using duality, we prove the following key lemma:
\begin{lemma}\label{lemma:bregman_young}
Let $x^+$ be such that $\nabla h(x^+) = \nabla h(x) - g$, and similarly define $x^+_1$ and $x^+_2$ from $g_1$ and $g_2$. Then, if $g = \frac{g_1 + g_2}{2}$, we obtain:
$$D_h(x, x^+) \leq \frac{1}{2}\left[ D_h(x, x^+_1) + D_h(x, x^+_2)\right].$$
\end{lemma}
Lemma~\ref{lemma:bregman_young} can be adapted for any $g = (1 - \alpha)g_1 + \alpha g_2$ with $\alpha \in [0, 1]$. In the Euclidean case $h=\|\cdot\|^2$, we recover $\|\frac{g_1 + g_2}{2}\|^2 \leq \frac{1}{2}\left(\|g_1\|^2 + \|g_2\|^2\right)$.  We now generalize the cocoercivity of the gradients~\citep[Eq. 2.1.7]{Nesterov2004} to the relatively smooth case:
\begin{lemma}[Bregman Cocoercivity]
\label{lemma:bregman_cocoercivity_main}
  If a convex function $f$ is relatively $L$-smooth w.r.t to $h$, then for any $\eta \leq \frac{1}{L}$, 
  \[ D_f(x,y) \geq \frac{1}{\eta} D_{h^*} \! \left(\nabla h(x) - \eta \left( \nabla f(x) - \! \nabla f(y) \right) ,\nabla h(x) \right), \]
  for all $x,y \in \interior \dom h$.
\end{lemma}

\subsection{Variance definition}
We start by specifying two assumptions on the structure of the noise. Note that we use a constant step-size ${\eta > 0}$ throughout this section for simplicity, but similar results hold with decreasing step-sizes. We denote $x^\star = \arg \min_x f(x)$ the minimizer of $f$ and $\|x\|_H^2 = x^\top H x$ for a positive definite operator $H$ and $x\in \R^d$.

\begin{assumption} \label{assumption:noise}
The stochastic gradients $g_t$ are such that $g_t = \nabla f_\xit (x_t)$, with $\espk{\xit}{f_\xit} = f$ and $f_\xit$ is convex and $L_{f/h}$-relatively smooth with respect to $h$ for all $\xi_t$. Besides, there exists a constant $\sigma^2 \geq 0$ such that:
\begin{align*}
    \sigma^2 &\geq \frac{1}{2\eta^2} \espk{\xit}{ D_{h^*}(\nabla h(x_t) - 2\eta \nabla f_\xit(x^\star), \nabla h(x_t))}\\
    &= \espk{\xi_t}{\|\nabla f_{\xi_t}(x^\star)\|^2_{\nabla^2 h^*(z_t)}},
\end{align*}
for some $z_t \in [\nabla h(x_t) - 2\eta \nabla f_{\xi_t}(x^\star), \nabla h(x_t)]$.
\end{assumption}

The assumption that the stochastic gradients are actual gradients of stochastic functions which are themselves smooth with respect to $h$ is rather natural, as already discussed in the introduction. It is at the heart of variance reduction in the finite sum setting (though the sum does not need to be finite in the case of Assumption~\ref{assumption:noise}), and is in particular verified when solving (Empirical) Risk minimization problems.

Yet, it prevents the analysis from applying to coordinate descent methods for instance, in which $g_t = \nabla_i f(x_t)$, with $i \in \{1, \cdots, d\}$. However, in this case, the extra structure can also be leveraged to obtain similar results~\citep{hanzely2018fastest, hendrikx2020dual, gao2020randomized}. 

For the variance, Assumption~\ref{assumption:noise} is a Bregman adaptation of the usual variance at the optimum definition used for instance in~\citet{moulines2011non, gower2019sgd}. Note that if $h^*$ is $\mu_h$-strongly convex with respect to the Euclidean norm, then the assumption is verified for instance when the variance is bounded in $\ell_2$ norm, as ${\|\nabla f_\xit(x^*)\|_{\nabla^2 h^*(z_t)}^2 \leq \mu_h^{-1} \|\nabla f_\xit(x^*)\|^2}$ (we used the fact that if $h$ is $\mu_h$-strongly convex, then $h^*$ is $1/\mu_h$-smooth, see e.g., \citet{kakade2009duality}).

We now compare our noise assumption with~\citep[Assumption 5.1.]{hanzely2018fastest}, which writes: 
\begin{equation}\label{eq:hanzely_variance}
    \frac{1}{\eta_t} \espk{\xit}{(\nabla f(x_t) - \nabla f_\xit(x_t))^\top (x_{t+1} - \bar{x}_{t+1})}\! \leq \! \sigma^2,
\end{equation}
for $t \geq 0$, where $g_t$ is the stochastic gradient estimate and $\bar{x}_{t+1}$ is the output of the (theoretical) Bregman gradient step taken with the true gradient, that is, $\nabla h(\bar{x}_{t+1}) = \nabla h(x_t) - \eta_t \nabla f(x_t)$. Thus, their condition can be written:
\[ \frac{1}{\eta_t^2} \espk{\xit}{D_h(x_{t+1}, \bar{x}_{t+1}) + D_h(\bar{x}_{t+1}, x_{t+1})} \leq \sigma^2,\]
so that $\sigma^2$ bounds at each step the distance (in the Bregman sense) between $x_{t+1}$ and $\bar{x}_{t+1}$, the point that would be obtained by the expected (deterministic) gradient update. To illustrate why our assumption is weaker, let us consider the case where $h$ is $\mu_h$-strongly convex. In this setting, a sufficient condition for \eqref{eq:hanzely_variance} to hold is that 
\begin{equation} \label{eq:hanzely_variance_sc}
    \frac{1}{\mu_h} \espk{\xit}{\| \nabla f(x_t) - \nabla f_{\xit}(x_t)\|^2} \leq \sigma^2,
\end{equation}
while a sufficient condition for our variance definition to hold is (using that $\nabla f(x^\star) = 0$):
\begin{equation} \label{eq:variance_sc}
    \frac{1}{\mu_h} \espk{\xit}{\| \nabla f(x^\star) - \nabla f_\xit(x^\star)\|^2} \leq \sigma^2,
\end{equation}
which only depends on the magnitude of the gradients at the optimum instead of the variance along the full trajectory since $x_t$ is replaced by $x^\star$. In particular, in the \textit{interpolation setting} where $\nabla f_{\xi}(x^\star) = 0$ for every $\xi$, $\sigma^2 = 0$ with our condition. Besides, if $f$ is strongly convex then the norm of its gradients increases when far from the optimum, and so one needs to restrict $x_t$ to a compact set of $\R^d$ for a condition such as~\eqref{eq:hanzely_variance_sc} to hold. In contrast, the condition from~\eqref{eq:variance_sc} can hold globally without further assumptions.

\subsection{Convergence results}
We now prove the actual convergence theorems for Bregman SGD. To avoid notation clutter, we generally omit with respect to which variable expectations are taken when clear from the context. 

\begin{theorem}\label{thm:sgd_strg_convex}
If $f$ is $L_{f/h}$-smooth and $\mu_{f/h}$-strongly convex relative to $h$ with $\mu_{f/h} >0$, and Assumptions \ref{assumption:blanket} and \ref{assumption:noise} hold, then for ${\eta \leq 1 / (2L_{f/h})}$, the iterates produced by Bregman stochastic gradient \eqref{eq:iter_sgd} satisfy
\begin{equation}
    \esp{D_h(x^\star, x_t)} \leq (1 - \eta \mu_{f/h})^t D_h(x^\star, x_0) +  \eta\frac{\sigma^2}{\mu_{f/h}}.
\end{equation}
\end{theorem}
Note that since we are in a Bregman setting, convergence is measured in terms of $D_h(x^\star, x_t)$, the distance between $x^\star$ and $x_t$ in the metric induced by $h$. If $h$ is $\mu_h$-strongly convex, then $D_h(x^\star, x_t) \geq \frac{\mu_h}{2}\|x_t - x^\star\|^2$ and convergence in $\ell_2$ distance is recovered.

\begin{proof}
By using Lemma~\ref{lemma:descent} from Appendix~\ref{app:sgd}
, we obtain:
\begin{align}
\label{eq:main_with_error_terms}
    \espk{\xit}{D_h(x^\star, x_{t+1})} &= D_h(x^\star, x_t) - \eta D_f(x^\star, x_t)  - \eta D_f(x_t, x^\star) + \espk{\xit}{D_h(x_t, x_{t+1})}.
\end{align}
Using Lemma~\ref{lemma:bregman_young}, the last term can be bounded as $D_h(x_t, x_{t+1}) \leq \frac{1}{2}\left[D_1 + D_2\right]$. We use Lemma~\ref{lemma:bregman_cocoercivity_main} (Bregman co-coercivity) to write:
\begin{align*}
    D_1 &= D_{h^*}(\nabla h(x_t) - 2\eta \left[\nabla f_{\xi_t}(x_t) - \nabla f_{\xi_t}(x^\star)\right], \nabla h(x_t))\leq 2 \eta D_{f_{\xi_t}}(x_t, x^\star),
\end{align*}
so that $\espk{\xit}{D_1 / 2} \leq \eta D_f(x_t, x^\star)$. Similarly,
\begin{equation}
     D_2 = D_{h^*}(\nabla h(x_t) - 2\eta \nabla f_\xi(x^\star), \nabla h(x_t)),
\end{equation}
so that $\espk{\xit}{D_2/2} \leq \eta^2 \sigma_t^2$. Thus, using the relative strong convexity of $f$ to bound the $D_f(x^\star, x_t)$ term, we obtain:
\begin{equation}
\label{eq:main_convex}
    \mathbb{E}_{\xit}D_h(x^\star, x_{t+1}) \leq (1 - \eta \mu_{f/h}) D_h(x^\star, x_t) + \eta^2 \sigma^2,
\end{equation}
which yields the desired result. 
\end{proof}

\begin{remark}[Interpolation]
    In the interpolation setting (when $\nabla f_{\xi_t}(x^\star) = 0$ for all $\xi_t$), we have that $\sigma^2 = 0$. Theorem~\ref{thm:sgd_strg_convex} thus proves linear convergence in this case. For instance, when solving objectives of the form $D_{\rm KL}(Ax, b)$ (which has applications in optimal transport~\citep{mishchenko2019sinkhorn}) or $D_{\rm KL}(b, Ax)$ (which has application in deblurring or tomographic reconstruction), then the variance as defined in~\citet{hanzely2018fastest} may be unbounded, whereas the variance as we define it is equal to $0$ if there exists $z$ such that $Az = b$.
\end{remark}

When $f$ is convex ($\mu_{f/h} = 0$), Theorem~\ref{thm:sgd_strg_convex} can be adapted to obtain a $1/T$ decrease of the error up to a noise region. 

\begin{theorem}[Convex case]\label{thm:sgd_convex}
Under the same assumptions as Theorem~\ref{thm:sgd_strg_convex}, if $\mu=0$, then
\begin{equation}\label{eq:thm_cvx_sgd}
    \esp{\frac{1}{T} \sum_{t=0}^T D_f(x^\star, x_t)} \leq \frac{D_h(x^\star, x_0)}{\eta T} + \eta \sigma^2
\end{equation}
\end{theorem}

Contrary to the Euclidean case, we do not obtain a guarantee on the average iterate in general. This is because the bound is on the average of $D_f(x^\star, x_t)$ instead of $D_f(x_t, x^\star)$, and Bregman divergences are not necessarily convex in their second argument (except for the Euclidean distance and Kullback-Leibler divergence). Therefore, the final bound is obtained on $\min_t D_f(x^\star, x_t)$, meaning that there is at least one $x_t$ such that this is true. Note that the nice properties regarding interpolation still hold in this setting. 

\begin{proof}
We start from Lemma~\ref{lemma:descent} and bound the $D_h(x_t, x_{t+1})$ in the same way as when $\mu > 0$, which yields:
\begin{equation*}
        \eta D_f(x^\star, x_t) = D_h(x^\star, x_t) - \espk{\xit}{D_h(x^\star, x_{t+1})} + \eta^2 \sigma^2.
\end{equation*}
Averaging over $t$ and dividing by $\eta$ leads to~\eqref{eq:thm_cvx_sgd}.
\end{proof}

The simplicity of the proof and the generality of our technical lemmas also allow us to provide convergence results when $f$ is not convex:

\begin{theorem}[Non-convex case]\label{thm:sgd_non_convex}
If $f$ is $L_{f/h}$-smooth relatively to $h$ and Assumptions \ref{assumption:blanket} and \ref{assumption:noise} hold, then for ${\eta \leq 1 / (2L_{f/h})}$, the iterates produced by Bregman stochastic gradient \eqref{eq:iter_sgd} satisfy
\begin{equation}\label{eq:thm_non_cvx_sgd}
    \esp{\frac{1}{T} \sum_{t=0}^T D_f(x^\star, x_t)} \leq \frac{D_h(x^\star, x_0)}{\eta T} + \eta \sigma^2.
\end{equation}
\end{theorem}

\section{Variance reduction}
\label{sec:variance_reduction}
We have shown in the previous section that BSGD enjoys guarantees that are similar to that of its Euclidean counterpart, although the notion of variance needs to be adapted. We show in this section that it is also possible to apply variance reduction to accelerate convergence. To this end, we solve for $n\in\mathbb{N}^*$ and some convex functions $f_i$:
\begin{equation} \label{eq:finite_sum_prob}
    \min_{x\in C} f(x):=\frac{1}{n}\sum_{i=1}^n f_i(x).
\end{equation}
The difference with Section~\ref{sec:sgd} is that we now assume that $f$ is a finite sum, which is required for variance reduction. We also assume that the minimizer $x^\star$ belongs to $\interior C$, so that $\nabla f(x^\star) = 0$. The case where $x^\star$ lies on the border of $C$ is more delicate, as $h$ might not be differentiable there (e.g., the log-barrier); this would require an involved technical analysis which we leave for future work.

To solve Problem~\eqref{eq:finite_sum_prob}, we consider Algorithm~\ref{algo:smd_vr}, which is a Bregman adaptation of the SAGA algorithm~\citep{Defazio2014}. Following its Euclidean counterpart, Algorithm~\ref{algo:smd_vr} stores the stochastic gradients computed at each iteration, and reuses them to estimate the full gradient. Therefore, only one stochastic gradient needs to be computed at each iteration, thus drastically reducing the iteration cost compared to batch gradient descent, at the expense of using more memory. Note that the stochastic updates are unbiased since $\espk{i}{g_t} = \nabla f(x_t)$, and at the optimum (when $x_t = \phi_i = x^\star$ for all $i$), $g_t = \nabla f(x^\star) = 0$ so the variance at the optimum is $0$ (contrary to SGD). We now study the convergence guarantees of Algorithm~\ref{algo:smd_vr} in more details. 

\begin{algorithm}[t]
\caption{Bregman-SAGA$( (\eta_t)_{t\ge 0}, x_0)$}
\label{algo:smd_vr}
\begin{algorithmic}[1]
\STATE $\phi_i=x_0$ for $i=1,...,n$
\vspace{0.5ex}
\FOR{$t=0,1,2,\ldots$}
\vspace{0.5ex}
\STATE Pick $i_t\in\{1,...,n\}$ uniformly at random
\vspace{0.5ex}
\STATE $g_t=\nabla f_{i_t}(x_t) - \nabla f_{i_t}(\phi^t_{i_t}) + \frac{1}{n} \sum_{j=1}^n \nabla f_j(\phi_j^t)$
\vspace{0.5ex}
\STATE $x_{t+1}=\arg \min_x \left\{\eta_t g_t^\top x + D_h(x, x_t)\right\}$
\vspace{0.5ex}
\STATE $\phi^{t+1}_{i_t} = x_t$, and store $\nabla f_{i_t}(\phi^{t+1}_{i_t})$.
\vspace{0.5ex}
\STATE $\phi^{t+1}_{j} = \phi_j^t$ for $j \neq i_t$.
\vspace{0.5ex}
\ENDFOR
\end{algorithmic}
\end{algorithm}

\subsection{Convergence Results}

For analyzing the Bregman-SAGA scheme, we first need to introduce, in addition to relative smoothness, an assumption on the regularity of $D_h$.
\begin{assumption}\label{assumption:regularity}
For all $i \in \{1, \cdots, n\}$, $f_i$ is $L_{f/h}$ relatively smooth w.r.t. $h$, and $f$ is $\mu_{f/h}$ relatively strongly convex w.r.t. $h$. Moreover, there exists a gain function $G$ such that for any $x,y,v \in \reals^d$ and $\lambda \in [-1,1]$, 
\begin{align*} 
  D_{h^*}\left(x + \lambda v, x \right) \leq G(x,y,v) \lambda^2 D_{h^*}\left(y + v, y \right).
\end{align*}
\end{assumption}
Such structural assumptions appear to be essential for analyzing Bregman-type methods that use information provided by gradients of past iterates. The function $G$ models the fact that the Bregman divergence $D_{h^*}(x+v,x)$ is not homogeneous nor invariant to translation in $x$ in general (except for the Euclidean case where it is equal to $\|v\|^2/2$). Note that such difficulties are also encountered for obtaining accelerated rates with inertial variants of Bregman descent, where similar assumptions are needed \cite{Hanzely2018}. This seems unavoidable, as suggested by the lower bound in \citet{dragomir2019optimal}.

Although the gain function $G$ is relatively abstract at this point, it plays a key role in defining the step-size, and convergence guarantees similar those of Euclidean SAGA can be obtained provided $G$ can be chosen small enough. 
We first state the general Theorem~\ref{thm:saga_G} (convergence proof for Algorithm~\ref{algo:smd_vr}), and then detail how $G$ can be bounded in several interesting cases. 

For $t\geq 0$ and step-sizes $\eta_t > 0$, define $H_t = \frac{1}{n}\sum_{i=1}^n D_{f_i}(\phi_i^t, x^\star)$, and the potential $\psi_t$  as follows:
\begin{equation}
    \psi_t= \frac{1}{\eta_t}D_h(x^\star, x_t) + \frac{n}{2} H_t.
\end{equation}
First note that by convexity of $h$ and of the $f_i$, $\psi_t \geq 0$ for all $t$. Our goal in this section is to show that $\{\psi_t\}_{t\geq 0}$ converges to $0$ at a given speed. Indeed, since $D_h(x^\star, x_t) \leq \psi_t$, this implies (as in Section~\ref{sec:sgd}) that $x_t$ converges to $x^\star$ at the same rate. To ease notations, we define
\begin{equation}
    \bar{\alpha}^t = \frac{1}{n}\sum_{j=1}^n \nabla f_j (\phi_j^t), \text{ and } \bar{\alpha}_i^t = \nabla f_i(\phi_i^t) - \bar{\alpha}^t.
\end{equation}

\begin{theorem}\label{thm:saga_G}
Assume that Algorithm \ref{algo:smd_vr} is run with a step size sequence $\{\eta_t\}_{t\geq 0}$ satisfying $\eta_t = 1 / (8L_{f/h} G_t)$ for every $t\geq 0$, with $G_t$ decreasing in $t$ and such that for all $j \in \{1, \cdots, n\}$:
\begin{equation*}
\begin{split}
    G_t \geq& G\left(\nabla h(x_t), \nabla h(x_t),\frac{1}{L_{f/h}}(\nabla f_j(x_t) - \nabla f_j(x^\star))\right),\\
    G_t \geq& G\Big(\nabla h(x_t) - 2 \eta_t \bar{\alpha}^t, \nabla h(\phi_j^t), \frac{1}{L_{f/h}}(\nabla f_j(\phi_j^t) - \nabla f_j(x^\star))\Big).
\end{split}
\end{equation*}
Then, under Assumptions \ref{assumption:blanket} and \ref{assumption:regularity}, the potential $\psi_t$ satisfies
\begin{equation}
    \espk{i_t}{\psi_{t+1}}\le \left(1-\min\left(\eta_t \mu_{f/h},\frac{1}{2n}\right)\right) \psi_t,
\end{equation}
In the convex case ($\mu_{f/h} = 0$), we obtain that
\begin{equation} \label{eq:main_saga_cvx}
    \esp{\frac{1}{4T} \sum_{t=1}^TD_f(x_t, x^\star) + H_t} \leq \frac{\psi_0}{T}.
\end{equation}
\end{theorem}

\begin{proof}[Proof] Similarly to BSGD, we apply Lemma~\ref{lemma:descent} (Appendix~\ref{app:sgd}), which yields
\begin{align}
\label{eq:main_with_error_terms_vr}
    \begin{split}
    \espk{i_t}{D_h(x^\star, x_{t+1})} =  D_h(x^\star, x_t) - \eta_t D_f(x^\star, x_t) 
    - \eta_t D_f(x_t, x^\star) + \espk{i_t}{D_h(x_t, x_{t+1})}.
    \end{split}
\end{align}
Lemmas~\ref{lemma:duality} and~\ref{lemma:bregman_young} yield $D_h(x_t, x_{t+1})\leq (D_1 + D_2)/2$, with 
\begin{align*}
    &D_1 = D_{h^*}(\nabla h(x_t) - 2\eta_t \left[\nabla f_i(x_t) - \nabla f_i(x^\star)\right], \nabla h(x_t)),\\
    &D_2 = D_{h^*}(\nabla h(x_t) - 2\eta_t( \nabla f_i(x^\star) - \bar{\alpha}_i^t), \nabla h(x_t)).
\end{align*}
Using Assumption~\ref{assumption:regularity} together with Lemma~\ref{lemma:bregman_cocoercivity_main}, we obtain: 
\begin{align*}
    D_1 &\leq 4 \eta_t^2 L_{f/h}^2 G_t \times  D_{h^*}(\nabla h(x_t) - \frac{1}{L_{f/h}}\left[\nabla f_i(x_t) - \nabla f_i(x^\star)\right], \nabla h(x_t))\\
    &\leq 4 \eta_t^2 L_{f/h} G_t D_{f_i}(x_t, x^\star).
\end{align*}
To bound the second term, we use Lemma \ref{lemma:breg_variance_standard}~\citep{pfau2013generalized} which is a Bregman version of the bias-variance decomposition. We write $V = 2 \eta_t \left[\nabla f_{i_t}(\phi_{i_t}^t) - \nabla f_{i_t}(x^\star)\right]$, so that $\espk{i_t}{V} = 2\eta_t \bar{\alpha}^t$ and: 
\begin{align*}
    &\espk{i_t}{D_2} =\espk{i_t}{D_{h^*}(\nabla h(x_t) - \espk{i_t}{V} + V, \nabla h(x_t))}\\
    &\leq \espk{i_t}{D_{h^*}(\nabla h(x_t) - \espk{i_t}{V} + V, \nabla h(x_t) - \espk{i_t}{V})}\\
    &\leq 4 \eta_t^2 L_{f/h}^2 G_t \mathbb{E}_{i_t} D_{h^*}\big( \nabla h(\phi_{i_t}^t) - L_{f/h}^{-1}\left[\nabla f_{i_t}(\phi_{i_t}^t) - \nabla f_{i_t}(x^\star)\right], \nabla h(\phi_{i_t}^t)\big)\\
    &\leq 4 \eta_t^2 G_t L_{f/h} \espk{i_t}{D_{f_i}(\phi_{i_t}^t, x^\star)}
\end{align*}
where we used the gain function for translation and rescaling the step size.
Following~\citet{hofmann2015variance}, we write:
\begin{equation} \label{eq:H_t_update}
    \espk{i_t}{H_{t+1}} = \left(1 - \frac{1}{n}\right)H_t + \frac{1}{n}D_f(x_t, x^\star).
\end{equation}
Therefore, we can use the $- H_t / n$ term to control the excess term from bounding $D_h(x_t, x_{t+1})$. In the end, we obtain:
\begin{align*}
    \espk{i_t}{\psi_{t+1} - \psi_t} \leq - D_f(x^\star, x_t) - \left(\frac{1}{2} - 2\eta_t L_{f/h} G_t \right) H_t -\left(1 - 2\eta_t L_{f/h} G_t - \frac{1}{2}\right)D_f(x_t, x^\star).
\end{align*}
If we choose $\eta_t \leq 1/(8 L_{f/h} G_t)$ then the last term is positive and $1 - 4\eta_t L_{f/h} G_t \geq 1/2$.
If $\mu_{f/h} > 0$ then we use the relative strong convexity of $f$ to obtain that the right hand side is proportional to $\psi_t$, thus leading to a linear convergence rate. Otherwise, we obtain a telescopic sum, leading to the $1/T$ rate of Equation~\eqref{eq:main_saga_cvx}.
\end{proof} 

Note that the monotonicity of $\eta_t$ (through $G_t$) is a technical condition to ensure that the Lyapunov is non-increasing. Otherwise, $\psi_t$ could blow up even though $x_{t+1}$ is very close to $x_t$, simply because $\eta_t$ shrinks. It could be replaced by the condition that $\eta_t$ does not vary too much (not more than a factor $1 - O(1/n)$), which achieves the same goal. The rest of this section is devoted to shong that non-trivial $G_t$ can be chosen in many cases, thus leading to strong convergence guarantees. In particular, the rate recovers that of Euclidean SAGA in case $h$ is a quadratic form. 

\begin{corollary}\label{cor:saga_quad}
If $\nabla^2 h$ is constant ($h$ is quadratic), then Assumption~\ref{assumption:regularity} is satisfied with $G = 1$, so that
\begin{equation}
    \esp{\psi_t}\le \left(1-\min\left(\frac{1}{8\kappa_{f/h}}, \frac{1}{2n}\right)\right)^t\psi_0,
\end{equation}
where $\kappa_{f/h} = L_{f/h} / \mu_{f/h}$ is the relative condition number. 
\end{corollary}

If $h$ is not quadratic, but $f^*$ and $h^*$ are regular with respect to a norm, then strong guarantees can also be obtained:

\begin{corollary} \label{thm:saga_cst_step}
If $h^*$ is $\mu_h^{-1}$-smooth and $f^*$ is $L_f^{-1}$-strongly convex with respect to a norm $\| \cdot \|^2$, then the stepsize can be chosen constant as $\eta_t = \frac{\mu_h}{8L_{f}}$, and
\begin{equation}
    \esp{\psi_t}\le \left(1- \min\left(\frac{\mu_h \mu_{f/h}}{8L_f}, \frac{1}{2n}\right)\right)^t\psi_0.
\end{equation}
\end{corollary}
Note that following~\citet{kakade2009duality}, having $h^*$ be $\mu_h^{-1}$-smooth is equivalent to having $h$ be $\mu_h$ strongly-convex.
\begin{proof}
The proof follows the same step as the proof of Theorem~\ref{thm:saga_G}, but the translation invariance and homogeneity are obtained by comparison with the norm, instead of using Assumption~\ref{assumption:regularity}. Thus, we pay a factor $\mu_h^{-1}$ when bounding $D_{h^*}$ by the norm, and a factor $L_f$ when bounding the norm by $D_{f^*}$. It is also possible to directly use Assumption~\ref{assumption:regularity}, but in this case the $L_f$ factor is replaced by $L_{f/h} L_h$, which is an upper bound on $L_f$, and may thus be slightly looser.
\end{proof}

Note that Corollary~\ref{cor:saga_quad} is actually a consequence of Corollary~\ref{thm:saga_cst_step}, since $\mu_h = 1$ and $L_f = L_{f/h}$ if $D_h$ is a norm itself. Otherwise, the constant $G_t$ is chosen in a rather pessimistic way, and depends on the difference between directly bounding $D_f$ by $D_h$ (in which case we pay a factor $L_{f/h}$), or going through a norm $\| \cdot \|$ in the middle (in which we case we pay $L_f / \mu_h \geq L_{f/h}$). 

As stated at the beginning of this section, one of the problems is that Bregman divergences lack translation  invariance and homogeneity. However, as the algorithm converges, one can expect these conditions to hold locally, as $D_{h^*}(x+v,x)$ is approximated by $\frac{1}{2}\|v\|^2_{\nabla^2 h^*(x^*)}$ for small enough $v$, and $x$ close enough to $x^*$.
This is indeed what happens under enough regularity assumptions on $h$.

\begin{proposition}\label{prop:gain_hess_lip}
    If $h$ is $L_h$-smooth and the Hessian $\nabla^2 h^*$ is $M$-smooth, then the gain function can be chosen as: 
    \[ G(x,y,v) = 1 + 2 M L_h \left(\|y-x\| + \|v\|\right) .\]
\end{proposition}
Note that, even if the regularity conditions of Proposition \ref{prop:gain_hess_lip} do not hold globally (such as for problems with unbounded curvature), they are at least valid on every bounded subset of $\interior C$, as soon as $h$ is $C^3$ on $\interior C$.
We now explicit a possible explicit choice for $G_t$ in this setting.
\begin{corollary} \label{corr:G_t}
    Assume that $h$ is $L_h$-smooth, $\mu_h$-strongly convex and that the Hessian $\nabla^2 h^*$ is $M$-smooth. Then, there exists an explicit constant $C$ such that if Algorithm \ref{algo:smd_vr} is run with a step size ${\eta_t = 1/(8L_{f/h}G_t)}$ with $G_t$ decreasing and satisfying
    \begin{equation}\label{eq:G_t_verifiable}
        \begin{split}
            G_t &\geq\min \Bigg(\frac{L_{f/h}L_h}{\mu_h},1  + C \Big( \sum_{j=1}^n \|x_t - \phi_j^t\| + \|\sum_{j=1}^n \nabla f_j(\phi_j^t)\|\Big)\Bigg),
        \end{split}
    \end{equation}
    then we have the convergence rate
    \begin{equation}
    \espk{i_t}{\psi_{t+1}}\le \left(1-\min\left(\frac{1}{8 G_t \kappa_{f/h}},\frac{1}{2n}\right)\right) \psi_t,
    \end{equation}
    where $\lim_{t \rightarrow \infty }G_t = 1$, or, more precisely,
    \begin{equation}\label{eq:G_t_psi_t}
        \esp{G_t} \leq 1 + \mathcal O \left(1 - \min\left(\frac{1 }{8 \kappa_h \kappa_{f/h}},\frac{1}{2n}\right)\right)^t.
    \end{equation}
\end{corollary}
The explicit expression for the constant $C$ is provided in Appendix~\ref{app:saga} along with the proof. Although the result involves smoothness constants of $h$ which can be large in the relatively-smooth setting, this dependence disappears asymptotically. Hence, after some time $t$, which we can roughly estimate using Equation~\eqref{eq:G_t_psi_t}, we obtain that $G_t = O(1)$. Thus, we reach the same kind of convergence rate as in the ideal quadratic case, which depends only on the \textit{relative} condition number $\kappa_{f/h}$, but with more general functions $h$, and thus possibly much better conditioning. Besides, the order of magnitude required for $G_t$ can be estimated during the optimization process using Equation~\eqref{eq:G_t_verifiable}.

\subsection{Remarks on adaptivity}
Assumption~\ref{assumption:regularity} highlights the fact that the key difficulty is purely geometric, and that in general we need to make up for the lack of translation invariance and homogeneity of Bregman divergences. Although Corollary~\ref{corr:G_t} gives a criterion for $G_t$ that can be evaluated throughout training (since the constant $C$ is explicit), several approximations are required to obtain it, and it may be loose overall. Yet, for the theory to hold, it suffices to have $\eta_t$ small enough such that: 
\begin{align*}
        \mathbb{E}_{i_t} D_h(x_t, x_{t+1}) \leq
    \frac{\eta_t}{4} \left[D_f(x_t, x^\star)  + \mathbb{E}_{i_t}D_{f_{i_t}}(\phi_{i_t}^t, x^\star)\right].
\end{align*}
Unfortunately, one would need to know $x^\star$ to evaluate such a condition, which is thus hard to use in practice. For the sake of clarity, we have only presented results for Bregman SAGA in this section. Yet, similar results hold for SVRG-style variance reduction, and we present them in Appendix~\ref{app:svrg}. An important difference is that in this case, $\phi_i^t = \phi_t$ for all $i$, and so the last term becomes $\esp{D_{f_i}(\phi_t, x^\star)|\mathcal{F}_t} = f(\phi_t) - f(x^\star)$ since $\nabla f(x^\star) = 0$, so we only need to know $f(x^\star)$ (or an estimation of it) in order to compute this criterion. In this case, we don't need to know the relative smoothness constant of the problem and the step-size can be set adaptively, similarly to~\citet{barre2020complexity}. Although it may be expensive to compute $D_f(x_t, x^\star)$ at each iteration, one can also approximate $f(x_t)$ on the fly, or only update $\eta_t$ periodically. On a side note, a similar criterion could be used for BSGD (without the second term in this case), in particular for over-parametrized problems for which we know that $f(x^\star) = 0$.

\section{Experiments}
\label{sec:experiments}

\begin{figure*}[ht!]
\subfigure[Poisson inverse problem (interpolation).]{
   \includegraphics[width=0.29\linewidth]{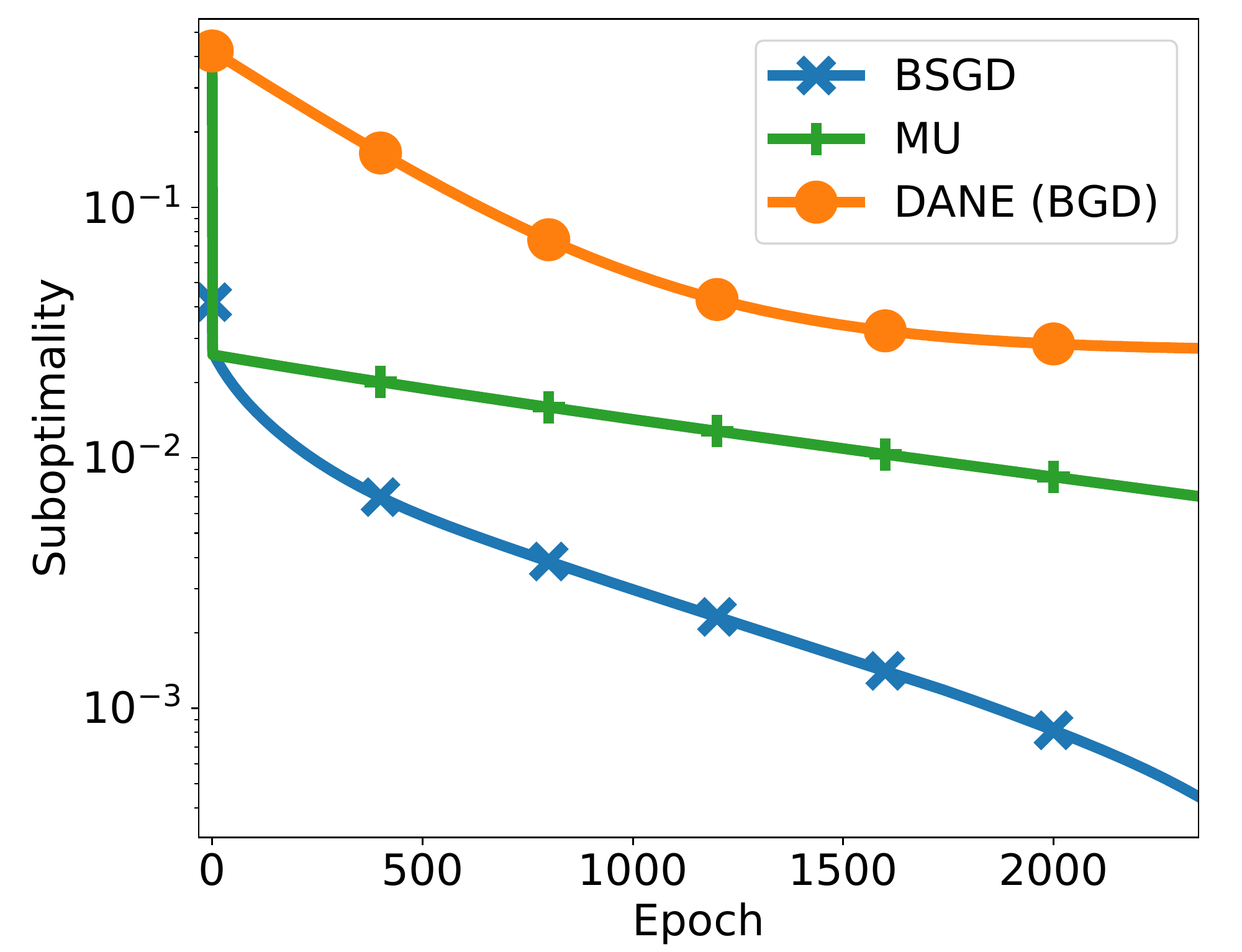}
    \label{fig:poisson_inverse}
}
\hfill
\subfigure[Tomographic reconstruction.]{
    \includegraphics[width=0.29\linewidth]{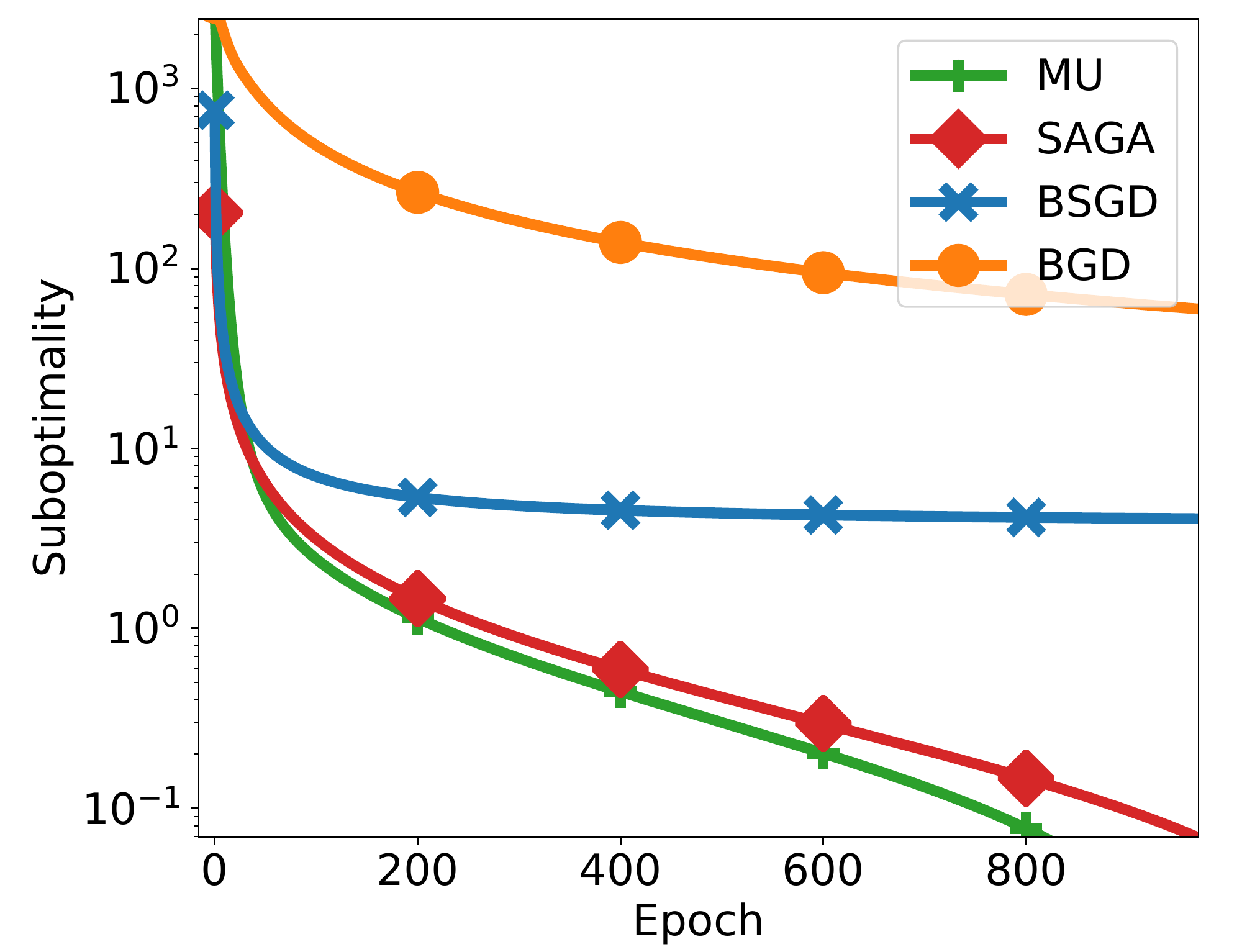}
    \label{fig:tomographic}
}
\hfill 
\subfigure[Distributed optimization.]{
    \includegraphics[width=0.29\linewidth]{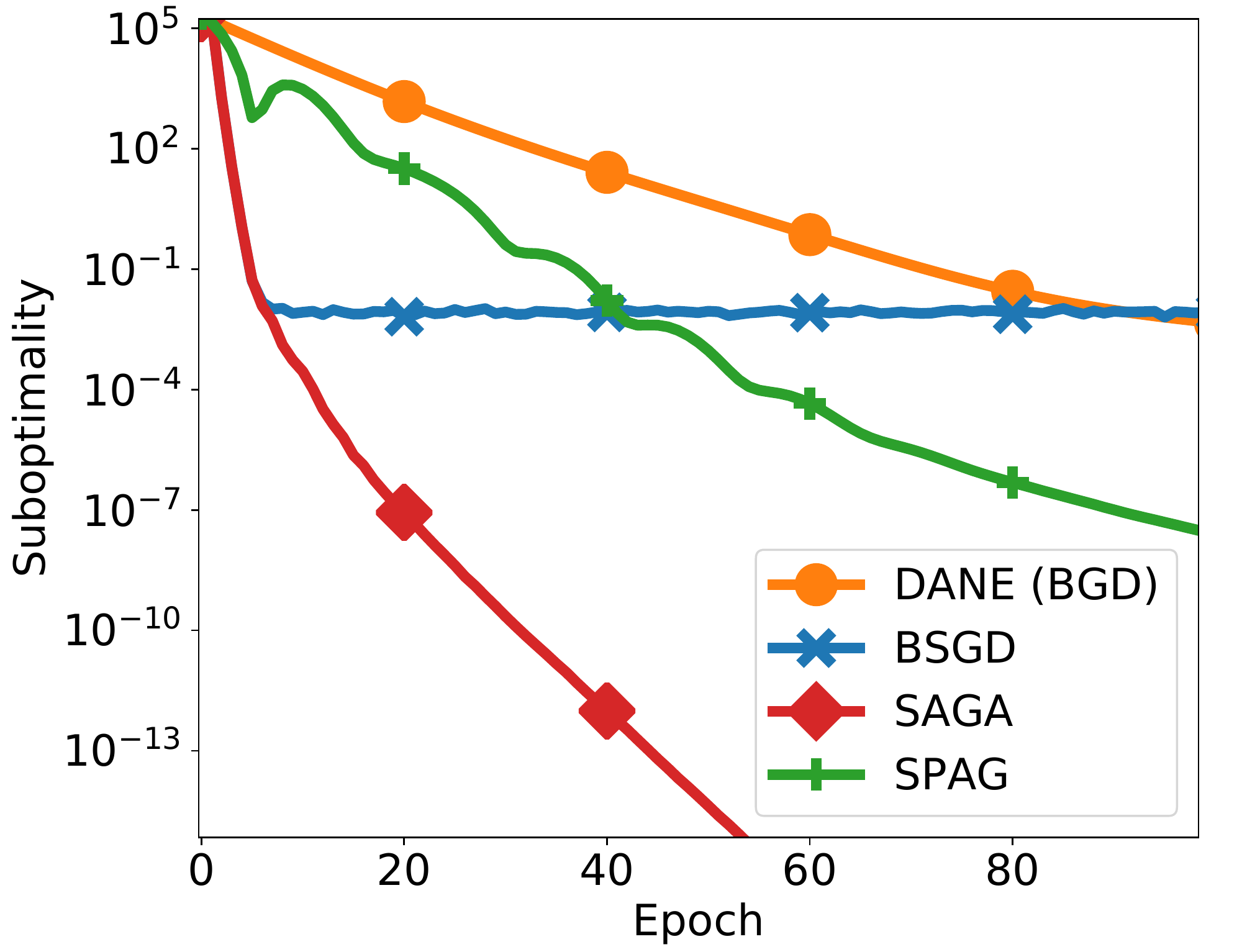}
    \label{fig:distrib_optim}
}\vspace{-10pt}
\caption{Bregman first-order methods on various applications.\label{fig:D_KL}
}
\end{figure*} 

In order to show the effectiveness of our method, we consider the two key settings mentioned in the introduction: problems with unbounded curvature (inverse problems with Poisson noise) and preconditioned distributed optimization. The first setting corresponds to the convex case ($\mu_{f/h} = 0$), whereas the second one corresponds to the relatively strongly convex case ($\mu_{f/h} > 0$). We observe that leveraging stochasticity (and, when needed, variance reduction) drastically improves the performance of Bregman methods in both cases. Additional details on the setting (such as the precise formulation of the objective or the relative smoothness constants) are given in Appendix~\ref{app:experiments}.

\subsection{Poisson inverse problems}
Figure~\ref{fig:poisson_inverse} considers the minimization problem $\min_{x \in \reals^{d}_{+}} f(x) = D_{\rm KL}(b, Ax)$, where $A \in \reals^{n \times d }_{+}$ and $D_{\rm KL}(u,v) = \sum_{i=1}^n u_i \log(u_i / v_i) - u_i + v_i$ is the Kullback-Leibler divergence. The goal is to recover an unknown signal $x_*$, observed through the matrix $A$ and corrupted by Poisson noise. This is a fundamental signal processing problem, with applications in astronomy and medicine (see \citet{Review2009} for a review). We use the log-barrier reference function, $h(x) = - \sum_i \log x_i$, for which relative smoothness holds with $L_{f/h}=\sum_{i=1}^n b_i/n$ \citep{Bauschke2017}.

We verify experimentally in this section that SGD is fast when the gradients at optimum are zero by first studying a problem where $b = Ax^\star$. $A \in \R^{n\times d}$ and $x \in \R^d$ are random (indices sampled uniformly between $0$ and $1$), with $n = 10000$ and $d=1000$. We compare the results of the deterministic and stochastic versions of Bregman Gradient descent. We also compare to the Multiplicative Updates (MU) algorithm, also known as Lucy-Richardson or Expectation-Maximization~\citep{MU-poisson}, which is a standard baseline for this problem. We observe that BGD is by far the slowest algorithm, but that BSGD is faster than Lucy-Richardson thanks to the stochastic speedup. We also observe that BSGD does not plateau in a noise region and converges to the true solution, which is consistent with Theorem~\ref{thm:sgd_convex}. The step-size for BGD and BSGD is chosen as $1 / L_{f/h}$, whereas Lucy-Richardson is parameter-free.

Figure~\ref{fig:tomographic} considers experiments on the tomographic reconstruction problem on the standard Shepp-Logan phantom~\citep{kak2002principles}. Due to space limitations, the main text mainly describes the results, but the setting details can be found in Appendix~\ref{app:experiments}. The step-size given by theory was rather conservative in this case, so we increased it by a factor of $5$ for all Bregman algorithms (and even 10 for BGD). Figure~\ref{fig:tomographic} shows again that stochastic algorithms drastically outperform BGD. Yet,  BSGD quickly reaches a plateau because of the noise. On the other hand, BSAGA enjoys variance reduction and fast convergence to the optimum. In this case, BSAGA is on par with MU, the state-of-the-art algorithm for this problem. This is because of the log barrier that allows relative smoothness to hold, but heavily slows down Bregman algorithms when coordinates are close to $0$. Yet, these results are encouraging and one may hope for even faster convergence of BSAGA for tomographic reconstruction with a tighter reference function. 

\subsection{Statistically Preconditioned Distributed Optimization}
In this section we consider the problem of solving a distributed optimization problem in which data is distributed among many workers. We closely follow the setting of~\citet{hendrikx2020dual}, and solve a logistic regression problem for the RCV1 dataset~\citep{lewis2004rcv1}. Function $h$ is taken as the same logistic regression objective as for the global objective $f$, but on a much smaller dataset of size $n_{\rm prec} = 1000$ and with an added regularization $c_{\rm prec} = 10^{-5}$. In this case, BGD corresponds to a widely used variant of DANE~\citep{shamir2014communication}, in which only the server performs the update. 
The stochastic updates in BSGD are obtained by subsampling a set of workers at each iteration, so that all the nodes do not have to participate in every iteration. Regularization is taken as $\lambda = 10^{-5}$, and there are $n=100$ nodes with $N=1000$ samples each. A fixed learning rate is used, and the best one is selected selected among $[0.025, 0.05, 0.1, 0.25, 0.5, 1.]$. BGD uses $\eta = 0.5$ while SAGA and BSGD use $\eta = 0.05$.
The x-axis represents the total number of communications (or number of passes over the dataset). Note that at each epoch, BGD communicates once with all workers (one round trip for each worker) whereas BSGD and BSAGA communicate $n$ times with one worker sampled uniformly at random each time. Therefore, BSAGA requires much less gradients from the workers to reach a given precision level, yet, it is at the cost of having to solve more local iterations.

Figure~\ref{fig:distrib_optim} first shows that BSAGA clearly outperforms BGD. BSGD on the other hand is as fast as BSAGA at the beginning of training, until it hits a variance region at which it saturates. This is consistent with the theory, and is similar to what can be observed in the Euclidean case. An interesting feature is that although the step-size has to be selected smaller than that of gradient descent (which is also the case in the Euclidean setting since $f$ is smoother than the least smooth $f_i$), choosing a constant step-size is enough to ensure convergence in this case, thus hinting at the fact that the analysis is rather conservative and that $G_t$ does not slow down the algorithm as much as we could have feared when far from the optimum. This is consistent with the results obtained by~\citet{hendrikx2020statistically} on acceleration.

\section{Conclusion}
\label{sec:conclusion}
Throughout the paper, we have (i) given tight convergence guarantees for Bregman SGD that allow to accurately describe its behaviour in the interpolation setting, and (ii) introduced and analyzed Bregman analogs to the standard variance-reduced algorithm SAGA. These convergence results require stronger assumptions on the objective than relative smoothness and strong convexity, but we show that fast rates can be obtained nonetheless when $h$ is nicely behaved (quadratic or Lipschitz Hessian). We also prove that these fast rates can be obtained for more general functions $h$ after a transient regime. Besides, we show experimentally that variance reduction greatly accelerates Bregman first-order methods for several key applications, including distributed optimization and tomographic reconstruction. In particular, there does not seem to be a slow transient regime in the applications considered, despite the lack of regularity of the objectives. This need for higher order regularity assumptions but great practical performance is consistent with  the results obtained for acceleration in the Bregman setting. Better understanding the transient regime (in which $G_t$ can be high) and finding better reference functions $h$ for the tomographic reconstruction problem are two promising extensions of our work.

{ \small
\section*{Acknowledgements} 
RD was supported by an AMX fellowship. RD would like to acknowledge support from the Air Force Office of Scientific Research, Air Force Material Command, USAF, under grant number  FA9550-19-1-7026/19IOE033 and FA9550-18-1-0226. HH was funded in part by the French government under management of Agence Nationale de la Recherche as part of the “Investissements d’avenir” program, reference ANR-19-P3IA-0001(PRAIRIE 3IA Institute). HH also acknowledges support from the European Research Council (grant SEQUOIA 724063) and from the MSR-INRIA joint centre.
}

\bibliography{library,library_radu}
\bibliographystyle{plainnat}

\newpage

\appendix

This appendix is organized as follows. We start by providing a detailed comparison of assumptions and convergence rates for our algorithms and related work in Figure~\ref{fig:conv_rates}. Then, Section \ref{app:sgd} provides the missing proofs for Bregman SGD, and Section \ref{app:saga} for the variance-reduced scheme Bregman SAGA. Additionally, we also analyze in Section \ref{app:svrg} another variant based on the SVRG algorithm. Finally, Section \ref{app:experiments} lists additional details for the numerical applications.

\renewcommand{\arraystretch}{1.7}
\setlength{\tabcolsep}{10pt}

\begin{figure}[h]
\centering
\begin{tabular}{|m{8em}|m{5em}|m{13em}|m{13em}|} 
 \hline 
 Algorithm & Gradient noise & Regularity assumptions  & Convergence rate of $D_h(x^*,x_t)$\\
 \hline
 Gradient Descent  & deterministic & $\mu I \preceq \nabla^2 f \preceq L I$ & $\mathcal{O}\left( 1 - \frac{\mu}{n L}\right)^t$  \\
 \hline
 Bregman Gradient Descent & deterministic & $\mu_{f/h} \nabla^2 h \preceq \nabla^2 f \preceq L_{f/h} \nabla^2 h$ & $\mathcal{O}\left( 1 - \frac{\mu_{f/h}}{n L_{f/h}} \right)^t$\\
 \hline
 Stochastic Gradient Descent & variance bounded at $x^\star$  & $\mu I \preceq \nabla^2 f_\xi \preceq L I$ & $\mathcal{O}\left( \left(1 - \frac{\mu}{L}\right)^t + \frac{\sigma^2 L}{\mu}\right)$  \\
 \hline
 Bregman Stochastic Gradient Descent (Theorem~\ref{thm:sgd_strg_convex})& variance bounded at $x^\star$ & $\mu_{f/h} \nabla^2 h \preceq \nabla^2 f_\xi \preceq L_{f/h} \nabla^2 h$ & $\mathcal{O}\left( \left(1 - \frac{\mu_{f/h}}{L_{f/h}} \right)^t + \frac{\sigma^2 L_{f/h}}{\mu_{f/h}} \right)$\\
 \hline
 SAGA \citep{Defazio2014} & finite sum & $\mu I \preceq \nabla^2 f_i \preceq L I$ & $\mathcal{O}\left( 1 - \min\left( \frac{1}{4n},\frac{\mu}{3L}\right)\right)^t$  \\
 \hline
 Bregman-SAGA, Corollary \ref{cor:saga_quad}& finite sum & $\mu_{f/h} \nabla^2 h \preceq \nabla^2 f_i \preceq L_{f/h} \nabla^2 h$, \vspace{1.5mm} \newline $\nabla^2 h$ constant & $\mathcal{O}\left( 1 - \min\left(\frac{1}{2n}, \frac{\mu_{f/h}}{8 L_{f/h}}\right) \right)^t$\\
 \hline
 Bregman-SAGA, Corollary \ref{thm:saga_cst_step}& finite sum & $\mu_{f/h} \nabla^2 h \preceq \nabla^2 f_i \preceq L_f I$, \vspace{1.5mm} \newline ${\mu_h I \preceq \nabla^2 h}$& $\mathcal{O}\left( 1 - \min\left(\frac{1}{2n}, \frac{\mu_{f/h} \mu_h}{8 L_f}\right) \right)^t$\\
 \hline
 Bregman-SAGA, Corollary \ref{corr:G_t} & finite sum & $\mu_{f/h} \nabla^2 h \preceq \nabla^2 f_i \preceq L_{f/h} \nabla^2 h$, \vspace{1.5mm} \newline $\mu_h I \preceq \nabla^2 h \preceq L_h I$,\vspace{1.5mm} \newline$\nabla^2 h^*$ is $M$-smooth & $\mathcal{O}\left( 1 - \min\left(\frac{1}{2n}, \frac{\mu_{f/h}}{8 G_t L_{f/h}}\right) \right)^t$\vspace{2mm}\newline with $G_t \rightarrow 1$ as $t \rightarrow \infty$\\
 \hline
\end{tabular}
\caption{Summary of convergence rates for standard (stochastic) first-order methods, and their Bregman counterparts in different settings.}
\label{fig:conv_rates}
\end{figure}

\section{Missing proofs for Bregman SGD (Section~\ref{sec:sgd})}
\label{app:sgd}
\begin{replemma}{lemma:bregman_young}
Let $x^+$ be such that $\nabla h(x^+) = \nabla h(x) - g$, and similarly define $x^+_1$ and $x^+_2$ from $g_1$ and $g_2$. Then, if $g = \frac{g_1 + g_2}{2}$, we obtain:
$$D_h(x, x^+) \leq \frac{1}{2}\left[ D_h(x, x^+_1) + D_h(x, x^+_2)\right].$$
\end{replemma}

\begin{proof}
By Lemma~\ref{lemma:duality} (duality), we have:
\begin{align*}
    D_h(x, x^+) &= D_{h^*}(\nabla h(x^+), \nabla h(x))\\
    &= D_{h^*}(\nabla h(x) - g, \nabla h(x))\\
    &= D_{h^*}\left(\frac{1}{2}[\nabla h(x) - g_1] + \frac{1}{2}[\nabla h(x) - g_2], \nabla h(x)\right)\\
    &\leq \frac{1}{2} D_{h^*}\left(\nabla h(x) - g_1, \nabla h(x)\right) + \frac{1}{2}D_{h^*}\left(\nabla h(x) - g_2, \nabla h(x)\right),
\end{align*}
where the inequality step is obtained by the convexity of the Bregman divergence in its first argument. The final result is obtained by using duality back.
\end{proof}

\begin{lemma} \label{lemma:descent}
If $\nabla h(x_{t+1}) = \nabla h(x_t) - \eta_t g_t$ with $\esp{g_t} = \nabla f(x_t)$, $\nabla f(x^\star) = 0$, then: \begin{equation}
            D_h(x^\star, x_{t+1}) = D_h(x^\star, x_t) - \eta_t D_f(x^\star, x_t)  - \eta_t D_f(x_t, x^\star) + D_h(x_t, x_{t+1}).
    \end{equation}
\end{lemma}

Note that this descent lemma is an equality, and we can then use standard assumptions to bound the different terms.
\begin{proof}
We start by writing $V_t(x) = \eta_t g_t^\top x + D_h(x, x_t)$. Since $x_{t+1}$ is defined as $\arg\min_x V_t(x)$ and by Assumption \ref{assumption:blanket}, we have $x_{t+1} \in \interior C$ then $\nabla V_t(x_{t+1}) = 0$ and so:
\begin{equation}
    V_t(x^\star) - V_t(x_{t+1}) = D_{V_t}(x^\star, x_{t+1}) = D_h(x^\star, x_{t+1}),
\end{equation}
since $\nabla^2 V_t = \nabla^2 h$. This writes:
\begin{equation}\label{eq:main_bregman}
    \eta_t g_t^\top (x^\star - x_{t+1}) + D_h(x^\star, x_{t}) - D_h(x_{t+1}, x_{t}) = D_h(x^\star, x_{t+1}).
\end{equation}
Then, we split the first term as $g_t^\top(x^\star - x_{t+1}) = g_t^\top(x^\star - x_{t}) + g_t^\top(x_t - x_{t+1})$. For the first term, we use the fact that $\esp{g_t} = \nabla f(x_t)$ and obtain
\begin{equation}\label{eq:breg_first_term}
    \esp{g_t^\top(x^\star - x_{t})} = - [\nabla f(x^\star) - \nabla f(x_t)]^\top(x^\star - x_t) = - D_f(x^\star, x_t) - D_f(x_t, x^\star),
\end{equation}
For the second term, we write:
\begin{align*}
    D_h(x_{t+1},x_t) + D_h(x_t, x_{t+1}) &= \langle \nabla h(x_t) - \nabla h(x_{t+1}), x_t - x_{t+1}\rangle\\
    &= \eta_t g_t^\top(x_t - x_{t+1}),
\end{align*}
so that 
\begin{equation}\label{eq:breg_second_term}
    \eta_t g_t^\top(x_t - x_{t+1}) - D_h(x_{t+1},x_t) = D_h(x_{t},x_{t+1}).
\end{equation}
Combining Equations~\eqref{eq:main_bregman}, \eqref{eq:breg_first_term} and~\eqref{eq:breg_second_term}, we obtain: 
\begin{equation}
    D_h(x^\star, x_{t+1}) = D_h(x^\star, x_t) - \eta_t D_f(x^\star, x_t)- \eta_t D_f(x_t, x^\star) + D_h(x_t, x_{t+1}),
\end{equation}
which finishes the proof.
\end{proof}

\begin{replemma}{lemma:bregman_cocoercivity_main}
  If a convex function $f$ is relatively $L$-smooth w.r.t to $h$, then for any $\eta \leq \frac{1}{L}$, 
  \[ D_f(x,y) \geq \frac{1}{\eta} D_{h^*} \! \left(\nabla h(x) - \eta \left( \nabla f(x) - \! \nabla f(y) \right) ,\nabla h(x) \right) \]
\end{replemma}

\begin{proof}
    Let $y \in \interior \dom h$ and consider the function $g_y$ defined by
    \[g_y(x) = D_f(x,y) = f(x) - f(y) -  \nabla f(y)^\top (x-y) \]
    for $x \in C$. $g_y$ is nonnegative, convex and relatively $L$-smooth with respect to $h$, since it has the same Hessian than $f$. Therefore, for $\eta \in (0,\frac{1}{L}]$ the relative smoothness inequality~\eqref{eq:rel_smooth_sc} implies that for every $u \in \interior \dom h$ we have $D_{g_y}(u,x) \leq \frac{1}{\eta} D_h(u,x)$, that is
    \begin{equation}
    g_y(u) \leq g_y(x) + \nabla g_y(x)^\top (u-x) + \frac{1}{\eta} D_h(u,x) := Q_y(u,x). 
    \end{equation}
    The right-hand side $Q_y(u,x)$ is a convex function of $u$ and is minimized for a point $u^+$ such that
    \begin{equation}
        \nabla h(u^+) - \nabla h(x) + \eta \nabla g_y(x) = 0, 
    \end{equation}
     and therefore
     \begin{align*}
         0 \leq g_y(u^+) &\leq Q_y(u^+,x)\\
         & = g_y(x) - \frac{1}{\eta}(\nabla h(u^+) - \nabla h(x))^\top (u^+ - x) + \frac{1}{\eta} D_h(u^+,x)\\
         &= g_y(x) -  \frac{1}{\eta} D_h(x,u^+)\\
         &= D_f(x,y) - \frac{1}{\eta} D_{h^*}\left(\nabla h(u^+),\nabla h(x)\right)\\
         &= D_f(x,y) - \frac{1}{\eta} D_{h^*}\left(\nabla h(x) - \eta \nabla g_y(x),\nabla h(x)\right)\\
     \end{align*}
     and the result follows from the fact that $\nabla g_y(x) = \nabla f(x) - \nabla f(y)$.
\end{proof}

\section{Missing proofs for Variance Reduced methods (Section~\ref{sec:variance_reduction})}\label{app:saga}
\subsection{Bregman variance decomposition}
First, we use the following Bregman counterpart of a standard variance identity \citep{pfau2013generalized}, which we prove for completeness.
\begin{lemma}[Bregman variance decomposition]
\label{lemma:breg_variance_standard}
    Let $X$ be a random variable on $\reals^d$. Then for any $u \in \reals^d$,
    \begin{equation}\label{eq:bregman_var_decomposition}
        \esp{ D_{h^*}(X, u)} = D_{h^*}(\esp{X},u) + \esp{D_{h^*}(X, \esp{X})}
    \end{equation}
\end{lemma}
As a consequence, for any random variable $V$ on $\reals^d$ and point $y \in \reals^d$ we have
\begin{equation}\label{eq:bregman_var_inequality}
      \esp{D_{h^*}(y+V - \esp{V}, y - \esp{V})} \geq \esp{ D_{h^*}(y + V - \esp{V}, y) }.
\end{equation}
\begin{proof}
Denoting $\overline{x} := \esp{X}$, We have for $u \in \reals^d$
\begin{align*}
     D_{h^*}(\overline{x},u) + \esp{D_{h^*}(X, \overline{x})} &= h^*(\overline{x}) - h^*(u) - \nabla h^*(u)^\top (\overline{x}-u)+ \esp{h^*(X) - h^*(\overline{x}) - \nabla h^*(\overline{x})^\top(X - \overline{x})}\\
     &= -h^*(u)  - \nabla h^*(u)^\top (\overline{x} - u) + \esp{h^*(X)} \\
     &= \esp{h^*(X) - h^*(u) - \nabla h^*(u)^\top(X - u)}\\
     &= \esp{D_{h^*}(X,u)}
\end{align*}
    which proves \eqref{eq:bregman_var_decomposition}. Then, \eqref{eq:bregman_var_inequality} follows from applying it to the point $u = y - \esp{V}$ and the random variable $X = y + V - \esp{V}$, along with using the nonnegativity of the Bregman divergence $D_{h^*}(\esp{X},u)$.
\end{proof}

\subsection{Proof of Theorem \ref{thm:saga_G}: generic Bregman-SAGA convergence bound}

In this subsection, we give a more detailed proof of Theorem~\ref{thm:saga_G}, and include derivations that had to be skipped in the main text because of space limitations. 

\begin{proof}[More detailed proof of Theorem~\ref{thm:saga_G}]
Similarly to BSGD, we start by applying Lemma~\ref{lemma:descent} (Appendix~\ref{app:sgd}), which yields
\begin{align}
\label{eq:main_with_error_terms_vr_app}
    \espk{i_t}{D_h(x^\star, x_{t+1})} =  D_h(x^\star, x_t) - \eta_t D_f(x^\star, x_t)   - \eta_t D_f(x_t, x^\star) + \espk{i_t}{D_h(x_t, x_{t+1})}.
\end{align}
Lemmas~\ref{lemma:duality} and~\ref{lemma:bregman_young} yield $D_h(x_t, x_{t+1})\leq (D_1 + D_2)/2$, with 
\begin{align*}
    &D_1 = D_{h^*}(\nabla h(x_t) - 2\eta_t \left[\nabla f_i(x_t) - \nabla f_i(x^\star)\right], \nabla h(x_t)),\\
    &D_2 = D_{h^*}(\nabla h(x_t) - 2\eta_t( \nabla f_i(x^\star) - \bar{\alpha}_i^t), \nabla h(x_t)).
\end{align*}
Using the gain function with the fact that $\eta_t \leq 1 / L_{f/h}$ and Lemma \ref{lemma:bregman_cocoercivity_main}, we have
\begin{equation}
\begin{split}
    \espk{i_t}{D_1} &= \espk{i}{D_{h^*}\left( \nabla h(x_t) - 2 \eta_t \left( \nabla f_i(x_t) - \nabla f_i(x^\star) \right) , \nabla h(x_t) \right)} \\
    & \leq 4 L_{f/h}^2\eta_t^2 \espk{i}{G\left(x_t,x_t, \frac{1}{L_{f/h}}(\nabla f_i(x_t) - \nabla f_i(x^\star))\right) D_{h^*} \left[ \nabla h(x_t) - \frac{1}{L_{f/h}} \left( \nabla f_i(x_t) - \nabla f_i(x^\star) \right), \nabla h(x_t) \right]}\\
    & \leq 4 L_{f/h} \eta_t^2 \espk{i}{G\left(x_t,x_t, \frac{1}{L_{f/h}}(\nabla f_i(x_t) - \nabla f_i(x^\star))\right) D_f(x_t,x^\star)} \\
    &\leq 4 L_{f/h} \eta_t^2 G_t D_f(x_t,x^\star).
\end{split}
\end{equation}
Note that we can pull the $G_t$ term out of the expectation over the choice of $i$ since $G_t$ holds for all $i$. For bounding $D_2$, Lemma \ref{lemma:breg_variance_standard} with $V = -2 \eta_t( \nabla f_i(x^\star) - \nabla f_i(\phi_i^t))$ leads to
\begin{equation}
\begin{split}
    \espk{i_t}{D_2} &= \espk{i}{D_{h^*}\left( \nabla h(x_t) - 2 \eta_t \left( \nabla f_i(x^\star) - \nabla f_i(\phi_i^t) + \frac{1}{n}\sum_{j=1}^n \nabla f_j(\phi_j^t) \right) , \nabla h(x_t) \right) } \\
    &\leq \espk{i}{ D_{h^*}\left( \nabla h(x_t) - 2 \eta_t \left( \nabla f_i(x^\star) - \nabla f_i(\phi_i^t) + \frac{1}{n}\sum_{j=1}^n \nabla f_j(\phi_j^t) \right) , \nabla h(x_t) - \frac{2 \eta_t}{n} \sum_{j=1}^n \nabla f_j(\phi_j^t) \right) }\\
    &\leq 4 \eta_t^2 L_{f/h}^2 \mathbb{E}_i\Bigg[ G\left( \nabla h(x_t) - \frac{1}{n} \sum_{j=1}^n \nabla f_j(\phi_j^t), \nabla h(\phi_j^t), \frac{1}{L_{f/h}}\left(\nabla f_i(\phi_j^t) - \nabla f_i(x^\star) \right) \right)\\
    &\qquad  D_{h^*}\left[ \nabla h(\phi_i^t) - \frac{1}{L_{f/h}}\left( \nabla f_i(\phi_i^t) - \nabla f_i(x^\star) \right), \nabla h(\phi_i^t) \right] \Bigg]\\
    &\leq 4 L_{f/h} \eta_t^2 G_t \espk{i}{D_{f_i}(\phi_i^t,x^\star)}.
\end{split}
\end{equation}
Recall that $H_t = \frac{1}{n}\sum_{j=1}^n D_{f_j}(\phi_j^t, x^\star)$. Plugging the expressions for $D_1$ and $D_2$ into Equation~\eqref{eq:main_with_error_terms_vr_app}, we obtain: 
\begin{align}
\label{eq:main_with_error_terms_vr_app_control}
    \espk{i_t}{D_h(x^\star, x_{t+1})} -  D_h(x^\star, x_t) \leq - \eta_t D_f(x^\star, x_t)   - \eta_t D_f(x_t, x^\star) + 2L_{f/h} \eta_t^2 G_t \left[D_f(x_t,x^\star) + H_t\right].
\end{align}
Following~\citet{hofmann2015variance}, we write:
\begin{equation} \label{eq:H_t_update_app}
    \espk{i_t}{H_{t+1}} = \left(1 - \frac{1}{n}\right)H_t + \frac{1}{n}D_f(x_t, x^\star),
\end{equation}
Indeed, $\phi_j^{t+1} = \phi_j^t$ with probability $1 - 1/n$, and $\phi_i^{t+1} = x_t$ with probability $1/n$. Therefore, we can use the $- H_t / n$ term to control the excess term from bounding $D_h(x_t, x_{t+1})$. 
In the end, using that $G_t$ is decreasing and so $\eta_t$ is increasing, we obtain the following recursion:
\begin{align} \label{eq:main_telescopic}
   \espk{i_t}{\psi_{t+1}} - \psi_t &= \frac{1}{\eta_{t+1}} D_h(x^\star, x_{t+1}) + \frac{n}{2}H_{t+1} - \frac{1}{\eta_t} D_h(x^\star, x_t) - \frac{n}{2}H_t  \nonumber \\ 
    &\leq \frac{1}{\eta_{t}}\left( D_h(x^\star, x_{t+1}) - D_h(x^\star, x_t)\right) + \frac{n}{2}\left(H_{t+1} - H_t\right) \nonumber\\ 
    &\leq  - D_f(x^\star, x_t) - \frac{1}{2}\left(1 - 4\eta_t L_{f/h} G_t \right) H_t  - \left(1 - 2\eta_t L_{f/h} G_t - \frac{1}{2}\right)D_f(x_t, x^\star).
\end{align}
If we choose $\eta_t \leq 1/(8 L_{f/h} G_t)$ then the last term is positive and $1 - 4\eta_t L_{f/h} G_t \geq 1/2$, so that using the relative strong convexity of $f$ leads to:
\begin{align*}
    \espk{i_t}{\psi_{t+1}} &\leq (\eta_t^{-1} - \mu_{f/h}) D_h(x^\star, x_t) + \left(1 - \frac{1}{2n} \right) \frac{n}{2} H_t\\
    &\leq \left(1 - \min\left(\eta_t \mu_{f/h}, \frac{1}{2n}\right)\right) \psi_t.
\end{align*}
The result can then be obtained by chaining this inequality. 
If $\mu_{f/h} = 0$ then we start back from Equation~\eqref{eq:main_telescopic}, use that $D_f(x^\star, x_t) \geq 0$ and the same fact that $1 - 4\eta_t L_{f/h} G_t \geq 1/2$ to obtain: 
\begin{equation*}
    \frac{1}{4}\left[D_f(x_t, x^\star) + H_t\right] \leq \psi_t - \espk{i_t}{\psi_{t+1}}.
\end{equation*}
The result is obtained by averaging over $T$, since the right hand side yields a telescopic sum, leading to the $1/T$ rate of Equation~\eqref{eq:main_saga_cvx}.
\end{proof}

\subsection{Lipschitz-Hessian setting}

In this section, we add the additional assumption that $h$ is $L_h$-smooth, and that the Hessian $\nabla^2 h^*$ is $M$-smooth in the operator norm, that is
\[ \|\left(\nabla^2 h^*(x) - \nabla^2 h^*(y)\right) u \| \leq M \|x-y\| \|u\| \]
for every $x,y,u \in \reals^d$.

\begin{repproposition}{prop:gain_hess_lip}
    If $h$ is $L_h$-smooth and the Hessian $\nabla^2 h^*$ is $M$-smooth, then the gain function can be chosen as: 
    \[ G(x,y,v) = 1 + 2 M L_h \left(\|y-x\| + \|v\|\right) .\]
\end{repproposition}

\begin{proof}[Proof of Proposition \ref{prop:gain_hess_lip}]
    Writing the divergence in integral form, we have for $x,y,v \in \reals^d$ and $\lambda \in [-1,1] $
    \begin{align*}
        D_{h^*}(x + \lambda v, x) &= \lambda^2 \int_0^1 \int_0^t v^\top \nabla^2 h^*(x + s \lambda v) v\, ds\, dt\\
        &\leq \lambda^2 \int_0^1 \int_0^t \left(v^\top \nabla^2 h^*(y + s v) v\, + M \|y + sv - x - \lambda s v\| \|v\|^2 \right) ds\, dt\\
        &\leq \lambda^2 \int_0^1 \int_0^t \left(v^\top \nabla^2 h^*(y + s v) v\, + M \left( \|y-x\| + 2s\|v\| \right) \|v\|^2 \right) ds\, dt\\
        &= \lambda^2 \left(D_{h^*}(y+v, y) + M (\|y-x\| + \|v\|) \|v\|^2 \right).
    \end{align*}
    Using the fact that is $h$ is $L_h$-smooth, $h^*$ is $1/L_h$-strongly convex and hence $\|v\|^2 \leq 2 L_h D_{h^*}(y+v,y)$, leading to
    \begin{align*}
        D_{h^*}(x + \lambda v, x) &\leq  \lambda^2 \left[ 1 + 2 M L_h \left(\|y-x\| + \|v\|\right)\right] D_{h^*}(y+v,y).
    \end{align*}
\end{proof}

\begin{repcorollary}{corr:G_t}
    Assume that $h$ is $L_h$-smooth and the Hessian $\nabla^2 h^*$ is $M$-smooth. Then, there exists an explicit constant $C$ such that if Algorithm \ref{algo:smd_vr} is run with a step size $\eta_t = 1/(8L_{f/h}G_t)$ with $G_t$ decreasing in $t$ and satisfying
    \begin{equation}
            G_t \geq \min \Bigg(\frac{L_{f/h}L_h}{\mu_h}, 1  + C \Big( \sum_{j=1}^n \|x_t - \phi_j^t\| + \|\sum_{j=1}^n \nabla f_j(\phi_j^t)\|\Big)\Bigg),
    \end{equation}
    then we have the convergence rate
    \begin{equation}\label{eq:cor3_conv_rate}
    \espk{i_t}{\psi_{t+1}}\le \left(1-\min\left(\frac{1}{8 G_t \kappa_{f/h}},\frac{1}{2n}\right)\right) \psi_t,
    \end{equation}
    where $\lim_{t \rightarrow \infty }G_t = 1$, or, more precisely,
    \begin{equation}
        \esp{G_t} \leq 1 + \mathcal O \left(1 - \min\left(\frac{1 }{8 \kappa_h \kappa_{f/h}},\frac{1}{2n}\right)\right)^t.
    \end{equation}
\end{repcorollary}

\begin{proof}[Proof of Corollary \ref{corr:G_t}]
Using the gain function from Proposition \ref{prop:gain_hess_lip}, to satisfy the assumptions of Theorem \ref{thm:saga_G} it is sufficient to choose $G_t$ such that
\begin{equation}\label{eq:gt_cond_cor3}
\begin{split}
    G_t & \geq 1 + 2M L_h \Big(\frac{1}{L_{f/h}} \|\nabla f_{i_t}(x_t) - \nabla f_{i_t}(x^\star)\| + \frac{1}{L_{f/h}} \| \nabla f_{i_t}(\phi_{i_t}^t) - \nabla f_{i_t}(x^\star) \| \\
    &\qquad + \| \nabla h(x_t) - \nabla h(\phi_{{i_t}}^t) - \frac{1}{4 n L_{f/h}}  \sum_{j=1}^n \nabla f_j(\phi_j^t)\| \Big).
\end{split}
\end{equation}
As the quantities involving $\nabla f_{i_t}(x^\star)$ are unknown, we provide an uper estimate. We can proceed in the following way, using the fact that, due to relative regularity, $f_i$ is also smooth with constant $L_{h} L_{f/h}$, and $f$ is strongly convex with constant $\mu_{h}\mu_{f/h}$:
\begin{align*}
    \|\nabla f_{i_t}(x_t) - \nabla f_{i_t}(x^\star)\|^2 & \leq 2L_h L_{f/h} D_{f_{i_t}}(x_t,x^\star)\\
    & \leq 2L_h L_{f/h}\, n \, D_f(x_t,x^\star) \\
    &\leq \frac{L_h L_{f/h}}{\mu_h \mu_{f/h}}\, n \, \| \frac{1}{n} \sum_{j=1}^n \nabla f_j(x_t) \|^2 \\
    & \leq \frac{\kappa_f \kappa_{f/h}}{ n} \left( \sum_{j=1}^n \| \nabla f_j(x_t) - \nabla f_j(\phi_j^t)\|  + \|\sum_{j=1}^n \nabla f_j(\phi_j^t)\| \right)^2\\
    & \leq \frac{\kappa_f \kappa_{f/h}}{ n} \left( \sum_{j=1}^n L_h L_{f/h}\|x_t - 
    \phi_j^t\|  + \|\sum_{j=1}^n \nabla f_j(\phi_j^t)\| \right)^2.
\end{align*}

And similarly, we can estimate the second term from
\begin{align*}
    \| \nabla f_{i_t}(\phi_{i_t}^t) - \nabla f_{i_t}(x^\star)\| &\leq \| \nabla f_{i_t}(x_t) - \nabla f_{i_t}(x^\star) \| + L_h L_{f/h} \|\phi_{i_t}^t - x_t\|,  \\
\end{align*}
which leads to the following upper estimate of the RHS of Condition \eqref{eq:gt_cond_cor3}:
\begin{align*}
     1 + &2ML_h\left( \frac{1}{L_{f/h}} \|\nabla f_{i_t}(x_t) - \nabla f_{i_t}(x^\star)\| + \frac{1}{L_{f/h}}\|\nabla f_i(\phi_{i_t}^t) -\nabla f_i(x^\star)\| \right.\\
     &+\left. \|\nabla h(x_t) - \nabla h(\phi_{i_t}^t) - \frac{1}{4n L_{f/h}} \sum_{j=1}^n \nabla f_j(\phi_j^t)\| \right) \\
     &\leq 1 + 2ML_h\Bigg( \frac{2}{L_{f/h}} \|\nabla f_{i_t}(x_t) - \nabla f_{i_t}(x^\star)\| + L_h \|\phi_{i_t}^t - x_t\| + \|\nabla h(x_t) - \nabla h(\phi_{i_t}^t) \| + \frac{1}{4n L_{f/h}} \| \sum_{j=1}^n \nabla f_j(\phi_j^t)\| \Bigg) \\
     &\leq 1 + 2M L_h \left( 2\sqrt{\frac{\kappa_h \kappa_{f/h}}{ n}} L_h\sum_{j=1}^n \|x_t - \phi_j^t\| +2L_h\|\phi_{i_t}^t - x_t\| + \left(\frac{1}{4n L_{f/h}} + \frac{2}{L_{f/h}}\sqrt{\frac{\kappa_h \kappa_{f/h}}{ n}} \right)\|\sum_{j=1}^n \nabla f_j(\phi_j^t)\|\right) \\
     &\leq 1 + C \left( \sum_{j=1}^n \|x_t - \phi_j^t\| + \|\sum_{j=1}^n \nabla f_j(\phi_j^t)\|\right)
\end{align*}
where $C$ is defined as
\[ C = 2 M L_h \max \left( 4 L_h \left( 1 + \sqrt{\frac{\kappa_h \kappa_{f/h}}{n}}\right), \frac{1}{L_{f/h}}\left( \frac{1}{4n} + 2 \sqrt{\frac{\kappa_h \kappa_{f/h}}{n}}\right) \right).\]

Now, with such choice of $G_t$, Theorem \ref{thm:saga_G} applies and the convergence rate \eqref{eq:cor3_conv_rate} holds. It remains to prove the estimate for the convergence rate of $G_t$ towards 1. To this end, we show that it is upper bounded by $\mathcal{O}(1+\psi_t^{1/2})$ since
\begin{equation}
    \begin{split}\label{eq:gt_leq_psi_t}
        1 + C \Bigg( \sum_{j=1}^n \|x_t - \phi_j^t\| &+ \|\sum_{j=1}^n \nabla f_j(\phi_j^t)\|\Bigg)\\
        &\leq  1 + C \left( \sum_{j=1}^n \|x_t - \phi_j^t\| + \|\sum_{j=1}^n \nabla f_j(x_t)\| + \sum_{j=1}^n \|\nabla f_j(\phi_j^t) -\nabla f_j(x_t)\|\right) \\ 
        &\leq  1 + C \left( \sum_{j=1}^n (1 + L_h {L_{f/h}}) \|x_t - \phi_j^t\| + n \|\nabla f(x_t)\|\right) \\ 
        &\leq  1 + C \left( \sum_{j=1}^n (1 + L_h {L_{f/h}})\left( \|x_t - x^*\| + \|x^* - \phi_j^t\|\right) + n L_h L_{f/h}\|x_t -x_*\|\right) \\ 
        &\leq  1 + C \left( n(1 + 2L_h L_{f/h}\|x_t-x^*\| + \sum_{j=1}^n (1 + L_h {L_{f/h}}) \|x^* - \phi_j^t\| \right) \\ 
        &\leq  1 + C \left( n(1 + 2L_h L_{f/h}) \sqrt{\frac{2}{\mu_h}D_h(x^*,x_t)} + \sum_{j=1}^n (1 + L_h {L_{f/h}}) \sqrt{\frac{2}{\mu_h \mu_{f/h}} D_{f_j}(\phi_j^t,x^*) } \right) \\ 
        & = 1 + \mathcal{O}\left( \sqrt{D_h(x^*,x_t)} + \sum_{j=1}^n \sqrt{D_{f_j}(\phi_j^t,x^*)} \right)\\
        & = 1 + \mathcal{O} \left(\sqrt{\psi_t}\right)
    \end{split}
\end{equation}
Since we imposed a safeguard such that $G_t \geq \frac{L_{f/h} L_h}{\mu_h}$, the convergence rate of $\psi_t$ is bounded by
\[ \esp{\psi_t} = \mathcal{O}\left( 1 - \min\left( \frac{1}{8\kappa_h \kappa_{f/h}}, \frac{1}{2n}\right) \right)^{t}\]
as stated by Corollary \ref{thm:saga_cst_step}. Indeed, the assumptions are verified as $h^*$ is $1/\mu_h$-smooth and $f^*$ is $1/L_f$-strongly convex with $L_f = L_h L_{f/h}$. This worst-case estimate for $\psi_t$, along with the majorization \eqref{eq:gt_leq_psi_t}, gives the resulting rate for $G_t$.
\end{proof}

\section{Bregman SVRG}
\label{app:svrg}

\begin{algorithm}[t]
\caption{Bregman-SVRG$( (\eta_t)_{t\ge 0}, x_0)$}
\label{algo:breg_svrg}
\begin{algorithmic}[1]
\STATE $\phi_0=x_0$, compute and store $\nabla f(\phi_0)$.
\vspace{0.5ex}
\FOR{$t=0,1,2,\ldots$}
\vspace{0.5ex}
\STATE Pick $i_t\in\{1,...,n\}$ uniformly at random
\vspace{0.5ex}
\STATE $g_t=\nabla f_{i_t}(x_t) - \nabla f_{i_t}(\phi_t) + \nabla f(\phi_t)$
\vspace{0.5ex}
\STATE $x_{t+1}=\arg \min_x \left\{\eta_t g_t^\top x + D_h(x, x_t)\right\}$
\vspace{0.5ex}
\STATE $   \phi_{t+1} =  \left\{
                \begin{array}{l}
 x_t, \text{ and compute and store } \nabla f(\phi_{t+1}) \text{ with probability } p\\
        \phi_t \text{ otherwise.}
        \end{array}
              \right.
$
\vspace{0.5ex}
\ENDFOR
\end{algorithmic}
\end{algorithm}

We consider in this section the convergence guarantees of Bregman SVRG (BSVRG), which is presented in Algorithm~\ref{algo:breg_svrg}. We consider the same variant as~\citet{hofmann2015variance}, in which the full gradient used for variance reduction is recomputed at each step with a small probability $p$, instead of after a fixed number of steps. We study this variant of BSVRG since it is very closely related to BSAGA. The main difference is that instead of updating $\phi_{i_t}$ when $i_t$ is picked, the algorithm chooses only one common $\phi_t$ to perform variance reduction, and this common $\phi_t$ is updated with probability $p$ at the end of each iteration. Thus, the convergence Theorem for Algorithm~\ref{algo:breg_svrg} closely follows Theorem~\ref{thm:saga_G}.

\begin{theorem}\label{thm:SVRG_G}
Assume that Algorithm \ref{algo:breg_svrg} is run with a step size sequence $\{\eta_t\}_{t\geq 0}$ satisfying $\eta_t = 1 / (8L_{f/h} G_t)$ for every $t\geq 0$, with $G_t$ decreasing in $t$ and such that for all $j \in \{1, \cdots, n\}$:
\begin{equation*}
\begin{split}
    G_t \geq& G\left(\nabla h(x_t), \nabla h(x_t),\frac{1}{L_{f/h}}(\nabla f_j(x_t) - \nabla f_j(x^\star))\right),\\
    G_t \geq& G\Big(\nabla h(x_t) - 2 \eta_t \nabla f(\phi_t), \nabla h(\phi_t), \frac{1}{L_{f/h}}(\nabla f_j(\phi_t) - \nabla f_j(x^\star))\Big).
\end{split}
\end{equation*}
Then, under Assumptions \ref{assumption:blanket} and \ref{assumption:regularity}, the potential $\psi_t = D_h(x^\star, x_t) + \frac{\eta_t}{2p} D_f(\phi_t, x^\star)$ satisfies
\begin{equation}
    \espk{i_t}{\psi_{t+1}}\le \left(1-\min\left(\eta_t \mu_{f/h},\frac{p}{2}\right)\right) \psi_t,
\end{equation}
In the convex case ($\mu_{f/h} = 0$), we obtain that
\begin{equation}
    \esp{\frac{1}{4T} \sum_{t=1}^T \eta_t \left[D_f(x_t, x^\star) + D_f(\phi_t, x^\star)\right]} \leq \frac{\psi_0}{T}.
\end{equation}
\end{theorem}

\begin{proof}
As explained before Theorem~\ref{thm:SVRG_G}, the only thing that changes between BSAGA and BSVRG is that a global $\phi_t$ is used instead of separate $\phi_i^t$, and that it is update with probability $p$ at the end of each iteration (instead of updating $\phi_{i_t}^t$ at time $t$ for SAGA). Thus, all the derivations performed for BSAGA hold for BSVRG if we replace $\phi_i^t$ with $\phi_t$ for all $i$. The only equation that needs to be adapted is Equation~\eqref{eq:H_t_update_app}, since it relies on the way the $\phi_i^t$ are updated. Yet, in the case of BSVRG, it writes:
\begin{equation} \label{eq:H_t_update_svrg}
    \esp{H_{t+1}} = \left(1 - p\right)H_t + p D_f(x_t, x^\star),
\end{equation}
which is the same as for BSAGA but with $p$ instead of $1/n$. Therefore, the conclusions are unchanged if we replace $n$ by $1/p$ whenever it appears in the bounds. Similar convergence guarantees hold when $\phi_t$ is updated every fixed number of steps $T$, but the proof is substantially more involved since Equation~\eqref{eq:H_t_update_svrg} does not hold in such a simple form. 
\end{proof}

\section{Additional details for the experiments}
\label{app:experiments}
Due to space limitations, some details of the experimental setting are missing from the main text, and we thus present them in this section. Note that all the experiments presented in this paper run in less than an hour on a standard laptop (and usually much less).  Our code is also available in supplementary material.   

\subsection{Poisson inverse problems}
We consider the minimization problem
\begin{equation}\label{eq:poisson_pb_app}
   \min_{x \in \reals^{d}_{+}} f(x) = \frac{1}{n}D_{\rm KL}(b, Ax) 
\end{equation}
 where $D_{\rm KL}(u,v) = \sum_{i=1}^n u_i \log(u_i / v_i) - u_i + v_i$ is the Kullback-Leibler divergence, and $A \in \reals^{n \times d}$ is a typically sparse matrix that models the measurement process. Problem \eqref{eq:poisson_pb_app} models the maximum likelihood estimation problem when assuming the statistical model
 \[ b \sim {\rm Poisson}(Ax^*)\]
 where $x^*$ is the true unknown signal. Inverse problems with Poisson noise arise in various signal processing applications such as astronomy or computerized tomography, see \citet{Review2009} and references therein. 
 
 As a motivating application of relative smoothness, \citet{Bauschke2017} prove that the Poisson objective $f$ is relatively smooth with respect to the log-barrier reference function
\[ h(x) = - \sum_{i=1}^d \log x_i\]
with constant $\sum_{j=1}^n b_j / n$. This constant can be quite conservative when $A$ is a sparse matrix, and so we prove a better estimate by leveraging this structure. For $j \in \{1\dots n\}$, we denote $S_j$ the support of the $j$-th column of $A$, that is
\[ S_j := \{ i \in \{1\dots n\} \, : A_{ij} \neq 0\}. \]

\begin{proposition}\label{prop:L_poisson}
    The Poisson objective function defined in \eqref{eq:poisson_pb_app} is relatively $L$-smooth w.r.t the log-barrier for 
    \begin{equation}
        L \geq \frac{1}{n} \, \max_{j \in \{1\dots d\}} \sum_{i \in S_j} b_i.
    \end{equation} 
\end{proposition}

\begin{proof} Let us denote $A_1,\dots A_n$ the row vectors of $A$. We refine the analysis from \citet[Lemma 7]{Bauschke2017} and start by writing for $x \in \reals^d_{++}, u \in \reals^d$
    \begin{equation*}
            u^\top \nabla^2 f(x) u = \frac{1}{n} \sum_{i = 1}^n b_i \frac{(A_i^\top u)^2}{(A_i^\top x)^2}.
    \end{equation*}
    Applying the Jensen inequality to the function $t \mapsto t^2$ and weights $w_{ij} = A_{ij} x_j / (A_i^\top x)$ yields
    \begin{equation*}
        \begin{split}
            d^\top \nabla^2 f(x) d & = \frac{1}{n} \sum_{i = 1}^n b_i \left( \sum_{j=1}^d w_{ij} \frac{u_j}{x_j} \right)^2 \\
            &\leq \frac{1}{n} \sum_{i=1}^n \sum_{j=1}^d b_i w_{ij} \cdot \frac{u_j^2}{x_j^2} \\
            & \leq \frac{1}{n} \sum_{j=1}^d \sum_{i \in S_j}^n b_i \frac{u_j^2}{x_j^2}\\
            & \leq L \sum_{j=1}^d \frac{u_j^2}{x_j^2}\\
            &= L \, u^\top \nabla^2 h(x) u
        \end{split}
    \end{equation*}
    where we used the fact that $w_{ij} \in [0,1]$ if $i \in S_j$, and $w_{ij} = 0$ otherwise.
\end{proof}
The relative Lipschitz constant provided by Proposition \ref{prop:L_poisson} can be considerably smaller than $\sum_{j=1}^n b_j /n$ when $A$ is sparse, which is the case in practical applications.

\begin{figure}
\centering
\subfigure[Original image]{
   \includegraphics[width=0.3\linewidth]{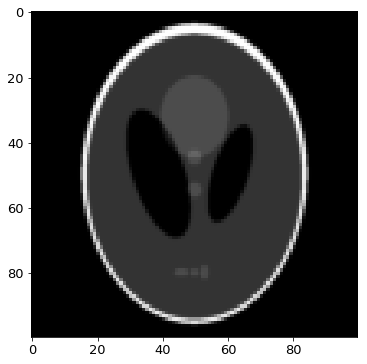}
    \label{fig:shepp-logan}
}
\subfigure[Sinogram]{
    \includegraphics[width=0.3\linewidth]{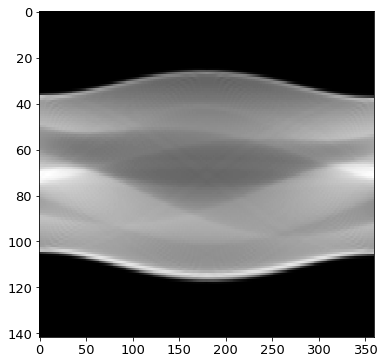}
    \label{fig:sinogram}
}
\caption{Illustration of the Radon transform on the Shepp-Logan phantom. On the sinogram, each column corresponds to the line integral of the image under different projection angles.\label{fig:tomograpy}
}
\end{figure} 

For our numerical experiments, we compare full-batch Bregman gradient descent (BGD), Bregman stochastic gradient descent (BSGD), and the Bregman SAGA scheme described in Algorithm \ref{algo:smd_vr}. We also implement the Multiplicative Update (MU), also known as Lucy-Richardson or Expectation-Maximization~\citep{MU-poisson}, which is the standard baseline for Poisson inverse problems.
 
 \paragraph{Synthetic problem in the interpolation setting.} In Figure~\ref{fig:poisson_inverse}, we simulate a synthetic problem the gradients at optimum are zero, by choosing $b = Ax^\star$ for some random $A \in \R^{n\times d}$ and $x^* \in \R^d$ (indices sampled uniformly between $0$ and $1$), with $n = 10000$ and $d=1000$.
 
 \paragraph{Tomographic reconstruction problem.} Computerized tomography~\citep{kak2002principles} is the task of reconstructing an object from cross-sectional projections, with fundamental applications to medical imaging. We study a classical synthetic toy problem for this task: the Shepp-Logan phantom (Figure \ref{fig:shepp-logan}). In this setting, the observation matrix $A$ corresponds to the discrete \textit{Radon transform}, which is the cross-sectional projection of the original image $x$ along different projection angles $\theta_1,\dots,\theta_n$ (Figure \ref{fig:sinogram}). That is, the objective writes
 
 \begin{equation}\label{eq:tomography_objective}
 f(x) = \frac{1}{n}D_{\rm KL}(b, Ax) = \frac{1}{n} \sum_{i=1}^n D_{\rm KL}(b_{\theta_i}, A_{\theta_i}x) \end{equation}
where $b_{\theta_i},A_{\theta_i}$ correspond to the observation and projection matrix along the angle $\theta_i$. For stochastic algorithms, the formulation \eqref{eq:tomography_objective} naturally yields a finite-sum structure: we thus take $f_i(x) = D_{\rm KL}(b_{\theta_i}, A_{\theta_i} x)$ for $i = 1\dots n$. 

We corrupt the sinogram with Poisson inverse noise, and apply our algorithms. We use $n = 360$ projection angles, and the image dimension is $d = 100^2$. As the matrix $A$ has a sparse structure, we use the relative smoothness constant provided by Proposition \ref{prop:L_poisson} for a better estimate. The step-size given by theory was rather conservative in this case, so we increased it by a factor of $5$ for all Bregman algorithms (and even 10 for BGD).

\subsection{Statistically Preconditioned Distributed Optimization}

We detail in this section the setting that was used to obtain Figure~\ref{fig:distrib_optim}. In particular, we use the following logistic regression objective with quadratic regularization, meaning that the function at node~$i$ is:
$$f_i: x \mapsto \frac{1}{N}\sum_{j=1}^N \log\left(1 + \exp(-y_{i,j} x^\top a_j^{(i)})\right) + \frac{\lambda}{2}\|x\|^2,$$
where $y_{i,j} \in \{-1, 1\}$ is the label associated with $a_j^{(i)}$, the $j$-th sample of node $i$. We use a regularization parameter of $\lambda = 10^{-5}$, and the size of the local datasets is equal to $N=1000$. The local dataset is constructed by shuffling the RCV1 dataset, downloaded from LibSVM, and then assigning a fixed portion to each worker. Then, one node (without loss of generality, node 0) uses its local dataset to construct the preconditioning dataset, so that:
\begin{equation}
    h: x \mapsto f_0(x) + \frac{c_{\rm prec}}{2}\|x\|^2,
\end{equation}
where $c_{\rm prec} = 10^{-5}$. Tuning $c_{\rm prec}$ in order to obtain the fastest algorithms is hard in general, as detailed in~\citet{hendrikx2020statistically} (in which it is denoted as $\mu$). One strategy is to choose $c_{\rm prec}$ of order $1 / n_{\rm prec}$ (in our case $n_{\rm prec} = N = 1000$), and then decrease it as long as BGD is stable. Our chosen value ($10^{-5}$) is smaller than that of~\citet{hendrikx2020statistically} for this problem ($10^{-4}$), in which they used a rougher $c_{\rm prec} = c / n_{\rm prec}$ criterion with varying $n_{\rm prec}$, and a larger step-size $\eta = 1$ for BGD (which is the same as DANE). Besides, we see that SPAG is slightly unstable in our example, and increasing $c_{\rm prec}$ would help with that. In this case, theory gives that $L_{f/h} \approx 1$. Yet, when $c_{\rm prec} \approx \lambda$, this step-size usually has to be chosen a bit smaller. Therefore, we choose in our case $\eta = 0.5$ for BGD and SPAG, and $\eta = 0.05$ for BSGD and BGD. Note that there is always a constant factor between the maximum step-size for SAGA and that of BGD, and the difference could further be explained by the difference between the batch condition number (relative smoothness of $f$) versus the stochastic one (max relative smoothness of the $f_i$).  

We compute the minimum error as the smallest error over all iterations for all algorithms. Then, we subtract it to the running error of an algorithm to get the suboptimality at each step. Following~\citet{hendrikx2020statistically}, local problems are solved using a sparse implementation of SDCA~\citep{shalev2016sdca}. We warm-start the local problems (initializing on the solution of the previous one), and perform 10 passes over the preconditioning dataset at each step, or until the norm of the gradient of the inner problem is small enough ($10^{-6}$). The number of inner passes could be reduced further, but then the algorithms started to converge slightly more slowly. This results in an overall computational overhead for the server, since BSAGA and BSGD require to solve many more inner problems, which are not so cheap to compute. Yet, this overhead only affects the server, and the iteration complexity is much lower, meaning that BSAGA is indeed very efficient to reduce the communication complexity of solving distributed empirical risk minimization problems.

\end{document}